\documentclass[11pt,leqno]{amsart}
\usepackage{amsmath}
\usepackage{amssymb}
\usepackage{amsthm}
\usepackage{enumerate}
\usepackage{graphics,epsfig}
\usepackage{rotating}

\renewcommand{\arraystretch}{1.2}
\newtheorem{theorem}{Theorem}[section]
\newtheorem{lemma}{Lemma}[section]
\newtheorem{corollary}{Corollary}[section]

\newtheorem{proposition}{Proposition}[section]

\newtheorem{remark}{Remark}[section]

\renewcommand{\arraystretch}{1.2}
\newtheorem{alg}{Algorithm}[section]
\newenvironment{algorithm}{
\begin{alg}
\begin{tabbing}
    1234\=5678\=9012\=3432\=2333\=232 \kill\\ %this line used for tabbing only
} { \end{tabbing}\end{alg}}

\def\ba{\begin{array}}
\def\ea{\end{array}}
\def\beq{\begin{equation}}
\def\eeq{\end{equation}}
\def\bea{\begin{eqnarray}}
\def\eea{\end{eqnarray}}
\def\beann{\begin{eqnarray*}}
\def\eeann{\end{eqnarray*}}

\def\cN{\mathcal{N}}

\def\rank{\textup{rank}}

\title{\bf Primal-Dual Entropy based
Interior-Point Algorithms for Linear Optimization\thanks{Research
was supported in part by Discovery Grants from NSERC.}}
\author{
Mehdi Karimi \and Shen Luo \and Levent Tun\c{c}el}
\date{\today}
\thanks{ Mehdi Karimi: (m7karimi@uwaterloo.ca)
Department of Combinatorics and Optimization, Faculty of
Mathematics, University of Waterloo, Waterloo, Ontario N2L 3G1,
Canada.\\
Shen Lou: (shenluo@alumni.uwaterloo.ca) \\ 
Levent Tun\c{c}el: (ltuncel@math.uwaterloo.ca)
Department of Combinatorics and Optimization, Faculty of
Mathematics, University of Waterloo, Waterloo, Ontario N2L 3G1,
Canada.}
\begin{document}

\begin{abstract}
We propose a family of search directions based on primal-dual
entropy in the context of interior-point methods for linear optimization. We show that
by using entropy based search directions in the predictor step of a predictor-corrector algorithm 
together with a homogeneous self-dual embedding, 
we can achieve the current best iteration complexity bound for linear optimization. Then, we focus on some   
wide neighborhood algorithms and show that in our family of entropy based search directions, we can find the 
best search direction and step size combination by performing a plane search at each iteration. For this purpose, we propose a heuristic plane search algorithm as well as
an exact one.  Finally, we perform computational experiments to study the performance of entropy-based search directions in wide neighborhoods of the central path, with and without utilizing the plane search algorithms.

\end{abstract}

\maketitle

\pagestyle{myheadings} \thispagestyle{plain}
\markboth{KARIMI, LUO, and TUN{\c C}EL}
{Interior-Point Algorithms Based on Primal-Dual Entropy}


%\noindent {\bf AMS Subject Classification:} ...

\section{Introduction}
Primal-dual interior-point methods have been proven to be one of the most useful algorithms in the area of modern interior-point methods for solving linear programming (LP) problems.  In this paper, we are interested in a class of path-following algorithms that generate a sequence of primal-dual iterates within certain neighbourhoods of the central path. Several algorithms in this class have been studied, which can be distinguished by the choice of search direction.  
We introduce a family of search directions inspired by  nonlinear reparametrizations of the central path equations, as well as the concept of entropy. Entropy and the underlying functions have been playing important roles in many different areas in mathematics, mathematical sciences, and engineering; such as partial differential equations \cite{dif-eqn}, information theory \cite{cover,shannon2001mathematical},  signal and image processing \cite{image,dusaussoy1991extended,nadeu1990flatness}, smoothing techniques \cite{PZ}, dynamical systems \cite{downarowicz2011entropy}, and various topics in optimization \cite{decarreau1992dual,erlander1977accessibility,erlander1981entropy,fang1997entropy,li1987entropy}. In the context of primal-dual algorithms, we use the entropy  function in determining the search directions as well as measuring centrality of
primal-dual iterates.   

Consider the following form of LP and its dual problem:
\begin{center}
\begin{tabular}{ccccccc}
\quad&(P)& minimize     &$c^{\top}x$  & \quad    &\quad\\
\quad&\quad& subject to & $Ax=b$,& $x\geq0$, \quad\\
\end{tabular}\\
\end{center}

\begin{center}\begin{tabular}{cccccc}
(D)&\quad&maximize&$b^{\top}y$&\quad&\quad \\
\quad&\quad&subject to&$A^{\top}y + s = c$,& $s \geq 0$&\quad\\
\end{tabular}
\end{center}
where $c \in {\mathbb R}^n$, $A \in {\mathbb R}^{m \times
n}$, and $b \in{\mathbb R}^m$ are given data. 	Without loss of generality, we always assume $A$ has full row rank, i.e., $\rank(A)=m$. Let us define ${\mathcal F}$ and ${\mathcal F_+}$ as 
\begin{eqnarray*} 
\mathcal F &:=& \{(x,s) : \ Ax=b, \ A^\top y + s = c, \ x\geq 0, s \geq 0, y \in \mathbb R^m\}, \\
\mathcal F_+ &:=& \{(x,s) : \ Ax=b, \ A^\top y + s = c, \ x >  0, s > 0, y \in \mathbb R^m\}.
\end{eqnarray*}
Next, we define the standard primal-dual central path with parameter $\mu >0$, i.e. $ \mathcal C:=\{(x_{\mu}, s_{\mu}): {\mu}>0 \}$, as the solutions of the following system:
\begin{eqnarray}  \label{KKT}
\begin{array}{ccrc}\\
\quad&\quad&A^{\top}y+s=c, & s >0 \\
\quad&\quad&Ax=b, & x >0 \\
\quad&\quad&Xs=\mu e,&\quad \\
\end{array}
\end{eqnarray}
where $e$ is the all ones vector whose dimension will be clear from
the context (in this case $n$). The above system has a unique solution for each $\mu >0$. For every pair $(x,s) \in \mathcal F_+$, we define the average duality gap as $\mu := \frac{x^Ts}{n}$.

In standard primal-dual 
algorithms, search direction is found by applying a Newton-like method to the equations in system \eqref{KKT} with an appropriate value of $\mu_+$ and the current point as the starting point. Explicitly, the search direction at a point $(x,s) \in \mathcal F_+$ is the solution of the following linear system of equations: 

\begin{eqnarray} \label{KKT-2}
\left [
\begin{array}{ccc}
0 & A^{\top} & I \\
A & 0   & 0 \\
S & 0   & X
\end{array} \right]\left[
\begin{array}{c}
d_x \\
d_y\\
d_s
\end{array} \right]
= \left[ \begin{array}{c}
0 \\
 0\\
 -XS e+ \mu_+ e\\
\end{array} \right].
\end{eqnarray}

The first two blocks of equations in \eqref{KKT} are linear and as a result, they are perfectly handled by Newton's method. The nonlinear equation $Xs=\mu_+ e$ plays a very critical role in Newton's method. Now, if we apply a continuously differentiable strictly monotone function $f: \mathbb R_+ \rightarrow \mathbb R$ to both sides of $Xs=\mu_+ e$ (element-wise), clearly the set of solutions of \eqref{KKT} does not change, but the solutions of Newton system might change dramatically.  This reparametrization of the KKT system can potentially give us an infinite number of search directions, but not all of them would have desirable properties. For a diagonal matrix $V$, let $f(V)$ and $f'(V)$ denote diagonal matrices with the $j$th diagonal entry equal to $f(v_j)$ and $f'(v_j)$, respectively.  Replacing $Xs=\mu_+ e$ with $f(Xs)=f(\mu_+ e)$ and applying Newton's method gives us the same system as \eqref{KKT-2} with the last equation replaced by (see \cite{TT97}):
\begin{eqnarray} \label{KKT-3}
Sd_x+Xd_s = (f'(XS))^{-1} (f(\mu_+ e) - f(Xs)). 
\end{eqnarray}
This kind of reparametrization has connections to \emph{Kernel functions} in interior-point methods (see our discussion in Appendix \ref{App1}). 

Every choice of a continuously differentiable strictly monotone function $f$ in  \eqref{KKT-3} gives us  a search direction. These search directions include some of the previously studied ones. For example, the choice of $f(x)=1/x$ gives the search direction proposed in \cite{Nazareth} (also see \cite{nazareth1997deriving} for another connection to the entropy function), and the choice of $f(x)=\sqrt{x}$ leads to the work in \cite{Darvay}. A natural choice for $f$ is $\ln(\cdot)$, which has been studied in \cite{TT97}, \cite{ZL2003}, and \cite{pan}. Substituting $f(x)=\ln(x)$ in \eqref{KKT-3} results in 
\begin{eqnarray} \label{KKT-4}
Sd_x+Xd_s = -(XS) \ln \left ( \frac{Xs}{\mu_+} \right ).
\end{eqnarray}  
This is the place where entropy function comes into play. In this paper, we study the behaviour of the search direction derived by using \eqref{KKT-4}, with an appropriate choice of $\mu_+$. This search direction corresponds to the gradient of the primal-dual entropy based potential function \mbox{$\psi(x,s):= \frac{1}{\mu} \sum_{j=1}^n x_js_j \ln(x_js_j)$} (see \cite{TT97}).
As in \cite{TT97}, we define a proximity measure $\delta(x,s)$ as:
\begin{eqnarray} \label{proximity}
\delta(x,s):=\sum_{j=1}^n\frac{x_js_j}{n\mu}\ln\left (\frac{x_js_j}{\mu} \right).
\end{eqnarray}
We sometimes drop $(x,s)$ in $\delta(x,s)$ when the argument of $\delta$ is clear from the context.    
If we choose $\mu_+$ such that $\ln\left ( \frac{\mu}{\mu_+}\right ) = 1- \delta(x,s)$, then \eqref{KKT-4} is reduced to 
\begin{eqnarray} \label{KKT-5}
Sd_x+Xd_s = -Xs+ \left [\delta Xs -(XS) \ln \left ( \frac{Xs}{\mu} \right ) \right ].
\end{eqnarray}  
This is exactly the search direction studied in \cite{TT97} for the following neighborhood (of the central path)
\[
{\mathcal N}_E(\beta):=\left \{ (x, s) \in {\mathcal F}_+: \frac{1}{2}-\beta \leq
\ln \left (\frac{x_js_j}{\mu} \right) \leq \frac{1}{2}+\beta, \text{ for all $j$}  \right \},
\]
where $\beta \geq \frac{1}{2}$. It is proved in \cite{TT97} that we can obtain the iteration complexity bound of $O(n \ln ( 1/\epsilon))$ for $\mathcal N_E(3/2)$. We will generalize \eqref{KKT-5} to define our family of entropy based search directions. 

%Before outlining the content of the paper, we discuss our choice of infeasible start algorithms. All interior-point methods start from a strictly feasible point and finding such a point has the same complexity as solving the LP itself. Many infeasible start algorithms have been  proposed so far, such as big $M$ method, Phase I-Phase II method, and modifying Newton system to approach feasibility and optimality simultaneously. To our knowledge, the current best iteration complexity bound for algorithms which modify the Newton system (in the first two blocks of equations in \eqref{KKT-2}) is $O\left (n \ln \left (1/\epsilon \right ) \right)$. 

In the vast literature on primal-dual interior-point methods, two of the closest treatments to ours are \cite{TT97} and \cite{ZL2003}. Our search directions unify and generalize the search directions introduced in \cite{TT97} and \cite{ZL2003}. Besides that, for infeasible start algorithms, we use homogeneous self-dual embedding proposed in \cite{YTM}. In this approach, we combine the primal and dual problems into an equivalent  homogeneous self-dual LP with an available starting point of our choice. It is proved in \cite{YTM} that we can achieve the current best iteration complexity bound of  $O \left (\sqrt{n} \ln  \left (1/\epsilon \right ) \right )$ by using this approach. See Appendix \ref{App2} for a definition of homogeneous self-dual embedding and the properties of it that we need. 

In Section \ref{prel}, we introduce our family of search directions that generalizes and unifies those proposed in \cite{TT97} and \cite{ZL2003}, and prove some basic properties. In Section \ref{PC-sec}, we use the entropy-based search direction in the predictor step of a predictor-corrector algorithm for the \emph{narrow neighborhood} of the central path
\[
{\mathcal N}_2(\beta):=\left \{ (x, s) \in {\mathcal F}_+: \left \|\frac{X
s}{\mu}-e \right \|_2 \leq \beta \right \},
\] 
and prove that we can obtain the current best iteration complexity bound of $O \left (\sqrt{n} \ln \left (1/\epsilon \right ) \right )$. After that, we focus on the \emph{wide neighborhood}
\[
{\mathcal N}_\infty^-(\beta) := \left \{ (x, s) \in {\mathcal F}_+:
\frac{x_js_j}{\mu} \geq
1-\beta, \text{ for all $j$}  \right \},
\]
and work with our new family of search directions, parameterized by $\eta$ (which indicates the weight of a component of the search direction that is based on primal-dual entropy). 
For various primal-dual interior-point algorithms utilizing the wide neighborhood, see \cite{MTY93,nesterov2008parabolic,potra2008primal,tunccel1995convergence} and the references therein. 
 In Section \ref{wide}, we derive some theoretical results for the wide neighborhood. However, our main goal in the context of wide neighborhood algorithms is to investigate the best practical performance for this class of search directions, in terms of total number of iterations. At each iteration, to find the best search direction in the family (i.e. the best value of $\eta$) that gives us the longest step (and hence the largest decrease in the duality gap),  we perform a plane search. For this purpose, we propose a heuristic plane search algorithm as well as an exact one in Section \ref{plane-search}. Then, in Section \ref{num}, we perform computational experiments to study the performance of entropy-based search directions with and without utilizing the plane search. Our computational experiments are on a class of classical small dimensional problems from NETLIB library \cite{netlib}. Section \ref{con} is the conclusion of this paper.

%%%%%%%%%%%%%%%%%%%%%

\section{ Entropy based search directions and Basic properties} \label{prel}

%For $\gamma=1$, the equations (CPS) define a $\emph{constant-gap
%centering direction}$, a Newton step toward the point $(x_\mu,
%y_\mu, s_\mu) \in$ $\mathcal C$, at which all the pairwise products
%$x_js_j$ are identical to $\mu$. Along this direction, the
%complementarily gap is kept constant. At the other extreme, the
%value $\gamma=0$ gives the standard Newton step, which
%is known as the $\emph{primal-dual affine-scaling}$, $direction$. \\
In this section, we derive some useful properties for analyzing our algorithms. It is more convenient to work in the scaled $v$-space. Let us define
\begin{eqnarray} \label{vu}
v &:=& X^{1/2} S^{1/2}e, \nonumber \\
u &:=& \frac{1}{\mu} Xs=\frac{1}{\mu} Vv. 
\end{eqnarray}
We define the scaled RHS vector with parameter $\eta \in \mathbb R_+$ as
\begin{eqnarray} \label{w}
w(\eta):= -v+ \eta \left [\delta v - V \ln \left ( \frac{Vv}{\mu}\right ) \right].
\end{eqnarray} 
This definition generalizes and unifies the search directions proposed in \cite{TT97} ($\eta=1$) and \cite{ZL2003} ($\eta=\frac 1 \sigma$ with $\sigma \in
(0.5,1)$ and $\sigma < \min \left \{1, \ln \left ( \frac{1}{1-\beta} \right\} \right )$). 
For simplicity, we write $w:=w(1)$, which is the scaled RHS vector of \eqref{KKT-5}. By using \eqref{vu}, we can also write
$\delta= \frac{1}{n} \sum_{j=1}^n u_j \ln(u_j)$. If we define $\left [w(\eta) \right]_p$ as the projection of $w(\eta)$ on the null space of the scaled matrix $\bar A:=  AD$, where $D:=X^{1/2}S^{-1/2}$, and define $\left [w(\eta) \right]_q:=w(\eta)-\left [w(\eta) \right]_p$, then in the original space, the primal and dual search directions are $d_x=D \bar d_x$ and $d_s = D^{-1} \bar d_s$, respectively, where $ \bar d_x := \left [w(\eta) \right]_p$ and $ \bar d_s :=\left [w(\eta) \right ]_q $. In other words, the scaled search directions, i.e., $\bar d_x$ and $\bar d_s$, can be obtained from the unique solution of the following system:

\begin{eqnarray} \label{SD}
\left [
\begin{array}{ccc}
0 & \bar A^{\top} & I \\
\bar A & 0   & 0 \\
I & 0   & I
\end{array} \right]\left[
\begin{array}{c}
\bar d_x \\
\bar d_y\\
\bar d_s
\end{array} \right]
= \left[ \begin{array}{c}
0 \\
 0\\
 w(\eta)\\
\end{array} \right].
\end{eqnarray}

Most of the upcoming results in this section are for the neighborhood ${\mathcal N}_\infty$ defined as:
\[
{\mathcal N}_\infty(\beta):= \left \{ (x, s) \in {\mathcal F}_+: \left \|\frac{X
s}{\mu}-e\right \|_\infty \leq \beta \right\}.
\]
We also use some of these results for ${\mathcal N}_2$ (since ${\mathcal N}_2 (\beta) \subset {\mathcal N}_\infty (\beta)$ for all $\beta >0$, this is valid).
Let us start with the following lemma (see \cite{TT97}):
\begin{lemma}
For every $x>0$, $s>0$, we have\\
1.  $\delta \geq 0$;\\
2.  equality holds above if and only if $Xs=\mu
e$.
\end{lemma}

The following lemma is well-known and is commonly used in the
interior-point literature and elsewhere.
\begin{lemma} \label{lem:ln}
For every $\alpha \in \mathbb R$ such that $|\alpha| \leq
1$, we have :\\
\beann\alpha-\frac{\alpha^2}{2(1-|\alpha|)} \leq \ln(1+\alpha) \leq
\alpha.\eeann
\end{lemma}

\begin{remark}
The RHS inequality above holds for every $\alpha \in (-1, +\infty)$.
\end{remark}

Next, we relate the primal-dual proximity measure $\delta(x,s)$ to a more commonly used 2-norm proximity measure for the central path. 

\begin{lemma}\label{lem:3.2}
Let  $\beta \in [0,1)$ such that $(x,s) \in {\mathcal N}_\infty (\beta)$.
Then,\\
\begin{align*}
\frac{1-3\beta}{2(1-\beta)n}\left\| \frac{Xs}{\mu}-e\right \|_2^2
\leq \delta(x,s) \leq \frac{1}{n}\left\|
\frac{Xs}{\mu}-e\right\|_2^2.
\end{align*}
\end{lemma}
\begin{proof}
The right-hand-side inequality was proved in \cite{TT97}. We prove
the left-hand-side inequality here. Let $\beta \in [0,1)$, such that
$(x,s) \in {\mathcal N}_\infty (\beta)$. Then, we have
(estimations are done in the $u$-space):
\begin{align*}
\delta(u)=\frac{1}{n}\sum_{j=1}^n u_j\ln\left
(u_j\right)&\geq
\frac{1}{n} \sum_{j=1}^n u_j\left[u_j-1-\frac{(u_j - 1)^2}{2(1-|u_j-1|)}\right ]\\
&=\frac{1}{n}\left \|u-e\right \|_2^2-\frac{1}{2n}
\displaystyle{\sum_{j=1}^n \frac{u_j}{\left(1-|u_j-1|\right)} \left(u_j-1\right)^2}\\
&\geq\frac{1}{n}\left\|u-e\right\|_2^2-\frac{(1+\beta)}{2n(1-\beta)}\left\|u-e\right \|_2^2\\
&=\frac{1-3\beta}{2(1-\beta)n}\left \|\frac{Xs}{\mu}-e\right \|_2^2.
\end{align*}
In the above, the first inequality uses Lemma \ref{lem:ln}, the second
inequality follows from the fact that $(x,s) \in {\mathcal N}_\infty (\beta)$.
\end{proof}

\begin{corollary}For every $(x, s) \in {\mathcal
N}_\infty \left (\frac{1}{4} \right)$, $\delta
\geq\frac{1}{6n}\left\|\frac{Xs}{\mu}-e\right \|_2^2$. Moreover,
for every $(x, s) \in {\mathcal N}_\infty  \left (\frac{1}{10} \right)$, $\delta
\geq \frac{7}{18n}\left\|\frac{Xs}{\mu}-e\right \|_2^2$.
\end{corollary}
%We also observe that the lower bound of $\delta$ approaches
%$\frac{1}{2n}\left \|\frac{Xs}{\mu}-e \right \|_2^2$ as $\left
%\|\frac{Xs}{\mu}-e\right \|_\infty$ goes to $0$.

%After the analysis of the proximity measure $\delta$, we think it is
%proper to review the search direction $w$ in another aspect, i.e.,
%we can decompose it into two orthogonal parts :
%affine-scaling and constant-gap centering.

Next, we want to study the behaviour of the search direction $w=-v+\delta v -V\ln(\frac{Vv}{\mu})$. We
already have upper and lower bounds on $\delta$, so we can easily
estimate $-v+\delta v$. Next, we estimate $V\ln(\frac{Vv}{\mu})$ within the neighborhood $\mathcal N_\infty(\beta)$.

\begin{lemma}\label{lem:3.3}Let $\beta \in [0, \frac{1}{2})$. Then, for every
$(x, s) \in {\mathcal N}_{\infty}(\beta)$, we have: \beann
\left(\delta(u)-2-\frac{\beta^2}{4\beta^2-6\beta+2}\right )v+\mu
V^{-1}e \leq w \leq (\delta(u)-2)v+\mu V^{-1}e.\eeann
\end{lemma}
\begin{proof}
Let $(x, s) \in {\mathcal N}_{\infty}(\beta)$ for some $\beta \in [0, \frac{1}{2})$. Then, $(1-\beta) e \leq \frac{Vv}{\mu} \leq (1+\beta) e$.

\noindent On the one hand, using Lemma \ref{lem:ln}, we have
\beann -V\ln\left(\frac{Vv}{\mu} \right )=V\ln(\mu
V^{-2}e)=V\ln(e+\mu V^{-2}e-e)\leq V(\mu V^{-2}e-e) = \mu V^{-1}e
-v.   \eeann On the other hand, using Lemma \ref{lem:ln} again and the fact
that $(x, s) \in {\mathcal N}_{\infty}(\beta)$, $\beta \in [0,
\frac{1}{2})$, for every $j \in \{1,2, \ldots, n\} $, we have
\[
v_j\ln \left ( \frac{\mu}{v_j^2} \right )= \sqrt{\mu u_j}
\ln\left(\frac{1}{u_j}\right) \geq \sqrt{\frac{\mu}{u_j}} -\sqrt{\mu
u_j} \left[1+\frac{\left(\frac{1}{u_j} -1\right)^2}
{2\left(1-\left|\frac{1}{u_j}-1\right|\right)}\right].
\]
To justify some of the remarks following this proof, we
focus on the cases \begin{itemize} \item $u_j \in [1-\beta, 1]$,
\item $u_j \in [1, 1+\beta].$
\end{itemize}
{\bf Case 1 $\left(u_j \in [1-\beta, 1]\right)$}: Using the
derivation above, we further compute\\
\begin{align*} v_j\ln \left ( \frac{\mu}{v_j^2} \right )
& \geq \sqrt{\frac{\mu}{u_j}} -\sqrt{\mu u_j}
\left[1+\frac{\left(\frac{1}{u_j} -1\right)^2}
{2\left(1-\left|\frac{1}{u_j}-1\right|\right)}\right]\\
& = \frac{\mu}{v_j} - v_j \left[1+ \frac{\left(1-u_j\right)^2}
{2u_j^2\left(2-\frac{1}{u_j}\right)}\right]\\
& = \frac{\mu}{v_j} - v_j \left[1+ \frac{\left(1-u_j\right)^2}
{2u_j\left(2u_j - 1\right)}\right]\\
& \geq \frac{\mu}{v_j} - v_j \left[1+ \frac{\beta^2}
{2(1-\beta)(1-2\beta)}\right].\\\end{align*} {\bf Case 2
$\left(u_j\in [1, 1+\beta]\right)$}: Again, using the derivation
before this case analysis, we further compute\\
\begin{align*}
v_j\ln \left ( \frac{\mu}{v_j^2} \right) & \geq \sqrt{\frac{\mu}{u_j}} -\sqrt{\mu
u_j} \left[1+\frac{\left(\frac{1}{u_j} -1\right)^2}
{2\left(1-\left|\frac{1}{u_j}-1\right|\right)}\right]\\
& =  \frac{\mu}{v_j} - v_j \left[1+ \frac{(1-u_j)^2} {2u_j}\right]\\
& = \frac{\mu}{v_j} - v_j \left(\frac{u_j^2+1}{2 u_j} \right)\\
& \geq \frac{\mu}{v_j} -\left[1+\frac{\beta^2}{2(1+\beta)}\right]
v_j.
\end{align*}

\noindent Therefore, within the neighborhood ${\mathcal N}_{\infty}(\beta)$, for
$\beta \in [0, \frac{1}{2})$, we conclude that the claimed
relation holds.\end{proof}

\begin{remark}
Focusing on the case analysis in the last proof, we see that for
those $j$ with $x_j s_j \geq \mu$ (Case 2), the corresponding
component $w_j$ of $w$ is very close to the corresponding component
computed for a generic primal-dual search direction.  For example,
for $\beta \in [0, 1/4]$,
\[
\left(\delta(u)-2-\frac{1}{40}\right)v_j + \frac{\mu}{v_j} \leq w_j
\leq \left(\delta(u)-2\right) v_j + \frac{\mu}{v_j}.
\]
\end{remark}
\begin{corollary}For every $(x, s) \in \mathcal{N}_\infty \left (\frac{1}{4} \right)$,
\beann \left(\delta(u)-2-\frac{1}{12}\right)v+  \mu V^{-1}e\leq w
\leq (\delta(u)-2)v+\mu V^{-1}e.\eeann
\end{corollary}

\begin{remark}
Recall that in a generic primal-dual search direction, $w$ is
replaced by $\left[-v +\gamma \mu V^{-1}e\right]$, $\gamma \in
[0,1]$ being the centering parameter. The above corollary shows that
inside the neighborhood $\mathcal{N}_\infty(\frac{1}{4})$,
\[
 \left[-1-\frac{1}{12(2-\delta(u))}\right]v+ \frac{1}{2-\delta(u)} \mu V^{-1}e\leq w \leq
-v+\frac{1}{2-\delta(u)}\mu V^{-1}e.
\]  Since by Lemma \ref{lem:3.2}, inside the neighborhood $\cN_{\infty}(1/4)$
we have $\delta(u) \leq 1/16$, working with $w$ is
close to setting the centrality parameter $\gamma :\approx \frac{1}{2}$.

\end{remark}

   Let us define the following quantities which play an important role in analysis of our algorithms:
    \beann
\Delta_{21}(u)  :=  \displaystyle{\sum_{j=1}^n u_j^2\ln(u_j)}, \ \ 
\Delta_{12}(u) :=  \displaystyle{\sum_{j=1}^n u_j\ln^2(u_j)}, \ \ 
\Delta_{22}(u)  :=  \displaystyle{\sum_{j=1}^n u_j^2\ln^2(u_j)}.
\eeann We drop the argument $u$, (e.g. we write $\Delta_{ij}$
instead of $\Delta_{ij}(u)$) when $u$ is clear from the context. The next few results provide bounds on the above quantities.

\begin{lemma}\label{lem:2.5}
Let $\beta \in [0,\frac{1}{4}]$ and assume that $(x, s)\in
{\mathcal N}_\infty(\beta)$. Then,\\
 \bea\xi_{ij} n\delta(u) \leq \Delta_{ij} \leq\zeta_{ij}
n\delta(u), \,\,\,\,\,\, \forall ij \in \{21, 22\}, \eea where\\
$\xi_{21} := 3(1-\beta)+2(1-\beta)\ln (1-\beta)$,\\
$\zeta_{21} := 3(1+\beta)+2(1+\beta)\ln (1+\beta)$,\\
$\xi_{22} := 2(1-\beta)+6(1-\beta)\ln (1-\beta)+6(1-\beta)\ln^2
(1-\beta) 	$, \\
$\zeta_{22} := 2(1+\beta)+6(1+\beta)\ln (1+\beta)+6(1+\beta)\ln^2
(1+\beta)$.
\end{lemma}
\begin{proof} See Appendix \ref{proofs}.
\end{proof}

\begin{corollary} \label{coro-1}
For every $(x, s) \in \mathcal{N}_\infty \left (\frac{1}{4} \right)$,
we have
\\
\beann  1.8 n \delta(u) \leq \Delta_{21} \leq \frac{9}{2} n
\delta(u), \ \ \text{and} \ \   \Delta_{22} < 5 n \delta(u). \eeann
%\beann  \sum_{i=1}^n (x_js_j)
%\ln^2\left(\frac{x_is_i}{\mu}\right)< 5 \sum_{i=1}^n u_i
%\ln(u_i)\eeann
\end{corollary}

\begin{lemma}\label{lem:3.4b}
Let $\beta \in [0, \frac{1}{2}]$ and assume that $(x,s) \in
{\mathcal N}_\infty^-(\beta)$. Then,\\
\beann 0  \leq \Delta_{12} \leq\zeta_{12} n\delta(u), \eeann
where
$\zeta_{12} := 2(\ln(n)+1)$. Furthermore, the upper bound is tight within a constant factor. 
\end{lemma}
\begin{proof}
The left-hand-side inequality obviously holds due to the
nonnegativity of the vectors $x$, $s$, $u$ and $\ln^2(Uu)$.
For $\zeta_{12}=2(\ln(n)+1)$, let us define $F_{12}:=\zeta_{12}n\delta(u)-\Delta_{12}$, then 
\begin{eqnarray*}
\nabla F_{12}(u)&=&2(\ln (n)+1)
e+2(\ln(n)+1)\ln(u)-\textup{Diag}(\ln(u))\ln(u)-2\ln(u), \\
\nabla^2 F_{12}(u)&=&2\ln(n)U^{-1}-2\textup{Diag}(\ln(u))U^{-1}.
\end{eqnarray*}

We consider the constrained optimization problem
$$\min_{u \in {\mathbb R}^n} F_{12}(u) \ \text{subject to} \ e^{\top}u-n=0, u-\frac{1}{2}e \geq 0.$$
The Lagrangian has the form $\mathcal
L_{12}(u,\lambda)=F_{12}(u)- \lambda_1(e^{\top}u-n)- \lambda_2^{\top}(u-\frac{1}{2}e )$.
Then, $\nabla^2 F_{12}$  is positive definite if $u<ne$.
Since we know that $u \leq \frac{n+1}{2}e$ within
$\mathcal{N}_\infty^-(\frac{1}{2})$, we conclude that
$F_{12}$ is strictly convex here. Moreover, for
$u^*=e$, the Lagrange multipliers
$\lambda^*_{1}=\zeta_{12}$ and $\lambda^*_2=0$ satisfy the  KKT conditions. Therefore, $u^*$ is the global minimizer of the
optimization problem. We notice that $F_{12}(u^*)=0$ which
implies the desired conclusion.

Let $u \in \mathbb R_{++}^n$ be a vector with $(n-1)$ entries equal to $1/2$ and one entry equal to $(n+1)/2$. Then we have:
\begin{eqnarray*}
\frac{\Delta_{12}}{n\delta(u)} = \frac{\frac{n-1}{2}\ln^2(1/2) + \frac{n+1}{2} \ln ^2 \left( \frac {n+1}{2} \right) }{\frac{n-1}{2}\ln(1/2) + \frac{n+1}{2} \ln \left(\frac{n+1}{2} \right) }
\leq  \ln \left( \frac{n+1}{2} \right)  = \ln (n+1) -1.
\end{eqnarray*}
Thus, the upper bound is tight within a constant factor.
\end{proof}

\begin{lemma}\label{lem:3.5}
Let $x>0$, $s>0$. Then $\Delta_{12} \geq n\delta^2$. Moreover, equality
holds if and only if $Xs=\mu e$.
\end{lemma}
\begin{proof}
Let $x>0$, $s>0$. Since $u_j > 0$, $\sqrt{u_j} \geq 0$ and $\sqrt{u_j}|\ln(u_j)|
\geq 0$. Using Cauchy-Schwarz inequality, we have \beann
\displaystyle{\sum_{j=1}^n u_j} \displaystyle{\sum_{j=1}^n
u_j\ln^2(u_j)} \geq \left (\displaystyle{\sum_{j=1}^n
u_j|\ln(u_j)|}\right )^2 \geq \left (\displaystyle{\sum_{j=1}^n
u_j\ln(u_j)}\right )^2.\eeann
Then, the claimed inequality follows. 
Moreover, by utilizing $\sum_{j=1}^n u_j =n$ we have equality if and only if $u=e$ (we used Cauchy-Schwarz inequality), or equivalently
$\frac{Xs}{\mu}=e$.
\end{proof}

Now, we have all the tools to state and analyze our algorithms.

%%%%%%%%%%%%%%%%%%%%%%%%%%%%%%%%%%%%%%%%%%%%%%%%%%%%%%%%%%%%%%%%%%%5
\section{Iteration Complexity
Analysis for Predictor-Corrector Algorithm} \label{PC-sec}
%%%%%%%%%%%%%%%%%%%%%%%%%%%%%%%%%%%%%%%%%%%%%%%%%%%%%%%%%%%%%%%%%%%%
As stated in previous section, our search directions are the solutions of system \eqref{SD}, where $w(\eta):= -v+ \eta \left [\delta v - V \ln \left ( \frac{Vv}{\mu}\right ) \right]$. Here, $\eta \in \mathbb R_+$ parameterizes the family of search directions. \cite{TT97} and \cite{ZL2003} studied these search directions for special $\eta$ from iteration complexity point of view. It is proved in \cite{TT97} that, using $w(1)$ as the search direction (i.e., $\eta=1$), we can obtain the iteration complexity bound of $O \left(n \ln (1/\epsilon)\right)$ for $\mathcal N_E(3/2)$, for feasible start algorithms.
These search directions have also been studied in \cite{ZL2003} in the wide neighborhood for the special case that $\eta=\frac 1 \sigma$ with $\sigma \in
(0.5,1)$ and $\sigma < \min \left \{1, \ln \left (\frac{1}{1-\beta} \right ) \right \}$. It was
shown in \cite{ZL2003} that the underlying infeasible-start algorithm utilizing a wide neighborhood has iteration complexity of
$O(n^2\ln \left ( \frac{1}{\epsilon} \right ) )$. 

In this section, we show that the current best iteration complexity bound $O(\sqrt{n}\ln \left ( \frac{1}{\epsilon} \right ))$ can be achieved if we use the entropy based search direction in the predictor step of the standard predictor-corrector algorithm   proposed by
Mizuno, Todd and Ye \cite{MTY93}, together with homogeneous self-dual embedding. 
Here is the algorithm:
\begin{algorithm} \label{PC}
Input: $(A, x^{(0)}, s^{(0)}, b, c, \epsilon)$, where $(x^{(0)},s^{(0)}) \in \mathcal N_2(\frac 14)$, and $\epsilon >0$ is the desired tolerance.\\
$(x,s) \leftarrow (x^{(0)},s^{(0)})$, \\ 
while $x^{\top}s>\epsilon$, \+\\
   \textbf{predictor step:} solve \eqref{SD} with $\eta=1$ for $\bar d_x$ and $\bar d_s$.\+\\
        $x(\alpha):=x+\alpha D {\bar d_x}$,\\
        $s(\alpha):=s+\alpha D^{-1} \bar d_s$, where $D=X^{1/2}S^{-1/2}$. \\
        $\alpha^*:=\max\{ \alpha:(x(\alpha),s(\alpha)) \in {\mathcal N}_2(\frac 12) \}.$  \-\\

    Let $x \leftarrow x(\alpha^*)$, and $ s \leftarrow s(\alpha^*)$ \\
    \textbf{corrector step:} solve \eqref{SD} for $\bar d_x$ and $\bar d_s$, where $w(\eta)$ is replaced by $-v+\mu V^{-1} e$, \\
    Let $x \leftarrow x+D \bar d_x$, $s\leftarrow s+D^{-1} \bar d_s$. \-\\
    end $\{$while$\}$.
\end{algorithm}

The $O\left(\sqrt{n}\ln \left ( 1/\epsilon \right) \right)$ iteration complexity bound is the conclusion of a series of lemmas. 
\begin{lemma} \label{LL1}
For every point $(x,s)\in {\mathcal N}_2(\frac{1}{4})$,  the following condition on $\alpha$ guarantees that  $(x(\alpha),s(\alpha))\in {\mathcal N}_2(\frac{1}{2})$:
\[ d_4\alpha^4+d_3\alpha^3+d_2\alpha^2+d_1\alpha+d_0 \leq 0, \]
where
\begin{eqnarray*}
d_0 &:=& -3\mu^2 \leq 0,\\
%d_1 &:=& 32\delta \mu^2 \beta^2+ 32n \delta \mu^2 -32\Delta_{21} \mu^2+6\mu^2
%\leq 2\delta \mu^2+2\mu^2+6\mu^2 \leq 9 \mu^2, \ \ \text{using Corollary \ref{coro-1}} \\
d_1 &:=& 32 \left (\delta \sum_{j=1}^n x_j^2 s_j^2  - n \delta \mu^2 -\Delta_{21} \mu^2 + n\delta \mu^2 \right)+6\mu^2
= 32\left (\delta \sum_{j=1}^n x_j^2 s_j^2  -\Delta_{21} \mu^2 \right )+6\mu^2, \\
d_2 &:=& 16 \left (\delta^2\sum_{j=1}^n(x_js_j)^2+\Delta_{22} \mu^2-2\delta \Delta_{21} \mu^2 +2C \right)-d_1-3\mu^2,\\
d_3 &:=& 32(\delta-1)C- 32B,\\
d_4 &:=& 16\sum_{j=1}^n(w_p)_j^2(w_q)_j^2, \\
B &:=& \sum_{j=1}^n x_js_j \ln \left (u_j\right)(w_p)_j(w_q)_j, \\ 
C &:=& \sum_{j=1}^n x_js_j(w_p)_j(w_q)_j.
 \end{eqnarray*}
\end{lemma}
\begin{proof}  See Appendix \ref{proofs}. 
 \end{proof}
 
 \begin{lemma} \label{LL2}
 For every point $(x,s)\in {\mathcal N}_2(\frac{1}{4})$, we have the following bounds on $d_1$, $d_2$, $d_3$ and $d_4$ defined in Lemma \ref{LL1}.
\begin{eqnarray*}
d_1 \leq 7\mu^2, \ \ \ \ d_2 \leq 20n\mu^2, \ \ \ \ d_3 \leq 64 n^{\frac 32} \mu^2, \ \ \ \ d_4 \leq 5n^2\mu^2. 
\end{eqnarray*}
 \end{lemma}
 \begin{proof}  See Appendix \ref{proofs}. 
 \end{proof}
 
We state the following well-known lemma without proof.
\begin{lemma} \label{LL3} \cite{MTY93} 
For every point $(x,s)\in {\mathcal N}_2(\frac{1}{2})$, the corrector step of Algorithm \ref{PC} returns a point in  the neighborhood ${\mathcal N}_2(\frac{1}{4})$. 
\end{lemma}
Now, we can prove the iteration complexity bound for Algorithm \ref{PC}. 
\begin{theorem} Algorithm \ref{PC} gives an $\epsilon$-solution in $O\left(\sqrt{n}\ln \left ( 1/\epsilon \right) \right)$ iterations.
\end{theorem}
\begin{proof}  
By Lemma \ref{LL1} and Lemma \ref{LL2}, in the predictor step, it is sufficient for $\alpha$ to satisfy 
\begin{eqnarray} \label{eq112} 
5n^2\alpha^4+64n^{\frac{3}{2}}\alpha^3+20 n\alpha^2+7\alpha \leq 3. 
\end{eqnarray}
It is easy to check that $\alpha=\frac{1}{50 \sqrt{n}}$ satisfies this inequality. Lemma \ref{LL3} shows that we have a point $(x,s)\in {\mathcal N}_2(\frac{1}{4})$ at the beginning of each predictor step and the algorithm is consistent. 
Since $x(\alpha)^{\top}s(\alpha)=(1-\alpha)x^{\top}s$ by part (b) of Lemma 3.1 of \cite{TT97}, we  deduce
that the algorithm will reach an $\epsilon$-solution in
$O\left(\sqrt{n}\ln( {1}/{\epsilon})\right)$ iterations.
\end{proof}

%%%%%%%%%%%%%%%%%%%%%%%%%%%%%%%%%%%%%%%%%%%%%%%%%%%%%%%%%%%%%%%%%%%5
\section{Algorithm for the Wide Neighborhoods} \label{wide}
%%%%%%%%%%%%%%%%%%%%%%%%%%%%%%%%%%%%%%%%%%%%%%%%%%%%%%%%%%%%%%%%%%%%%
In the rest of the paper, we study the behaviour of the entropy based search directions in a wide neighborhood. As mentioned before, for each $\eta$, our search direction is derived from
the solution of system \eqref{SD}, where $w(\eta):= -v+ \eta \left [\delta v - V \ln \left ( \frac{Vv}{\mu}\right ) \right]$. 
These search directions have been studied in \cite{ZL2003} in the wide neighborhood for the special case that $\eta=\frac 1 \sigma$ with $\sigma \in
(0.5,1)$ and $\sigma < \min \left \{1, \ln \left ( \frac{1}{1-\beta} \right ) \right \}$. In this paper, we study these search directions for a wider range of $\eta$. We prove some results on iteration complexity bounds in this section. However, in the rest of the paper, we mainly focus on the practical performance of our search directions in the wide neighborhood.

The algorithm in a wide neighborhood (with a value of $\eta \geq 0$ fixed by the user) is:
\begin{algorithm}  \label{wide-1}
Input $(A, x^{(0)}, s^{(0)}, b, c, \epsilon, \eta)$, $\epsilon >0$ is the desired tolerance.\\
$(x,s) \leftarrow (x^{(0)},s^{(0)})$, \\
while $x^{\top}s>\epsilon$ \+\\

   solve \eqref{SD} for $\bar d_x$ and $\bar d_s$,\+\\
        $x(\alpha):=x+\alpha D {\bar d_x}$,\\
        $s(\alpha):=s+\alpha D^{-1} \bar d_s$, where $D=X^{1/2}S^{-1/2}$. \\
        $\alpha^*:=\max\{ \alpha:(x(\alpha),s(\alpha)) \in {\mathcal N}_\infty^-(\beta) \}.$  \-\\

    Let $x \leftarrow x(\alpha^*); \; s \leftarrow s(\alpha^*)$ \-\\
end $\{$while$\}$.
\end{algorithm}

\begin{lemma}
\label{lem:5.1}
In Algorithm \ref{wide-1}, for every choice of $\eta \in \mathbb R_+$, we have $x(\alpha)^{\top} s(\alpha) =\left(1-\alpha\right)n\mu$.
\end{lemma}
\begin{proof}
We proceed as in the proof of Lemma 3.1 of \cite{TT97}, part (b):
\[
x(\alpha)^\top s(\alpha)=x^\top s+ \alpha v^\top (\bar d_x+ \bar d_s)= x^\top s + \alpha v^\top w(\eta) = (1-\alpha) x^ \top s.
\]
For the last equation, we used the facts that $v^\top v = x^\top s$ and $v$ and $\delta v-V \ln \left( \frac{Vv}{\mu}\right)$ are orthogonal. 
\end{proof}
This lemma shows that the reduction in the duality gap is independent of $\eta$ and is exactly the same as in the primal-dual affine scaling algorithm. So, Lemma \ref{lem:5.1} includes part (b) of Lemma 3.1 of \cite{TT97} and part (c) of Theorem 3.2 of \cite{monteiro} as special cases.  We show later that by performing a plane search, we can find an $\eta$ that gives the largest value of $\alpha$ in the algorithm (and hence the largest possible reduction in duality gap, per iteration).

\begin{lemma}
Let $x>0$, $s>0$. For $\eta \geq 0$, $\|w(\eta)\|_2^2=n\mu[1-\eta^2(\delta^2-\frac{\Delta_{12}}{n})]$.
\end{lemma}
\begin{proof}
Let $x>0$, $s>0$, and $\eta \geq 0$. Then, 
 \begin{align*}\|w(\eta)\|_2^2&=\sum_{j=1}^n v_j^2\left (\delta \eta-1-\eta \ln \left ( \frac{x_js_j}{\mu} \right ) \right )^2\\
 &=\sum_{j=1}^n x_js_j\left (\delta^2\eta^2+1+\eta^2 \ln^2 \left(\frac{x_js_j}{\mu} \right)+2\eta \ln \left (\frac{x_js_j}{\mu} \right)-2\delta\eta-2\delta\eta^2 \ln \left (\frac{x_js_j}{\mu} \right )\right )\\
 &=n\mu\delta^2\eta^2+n\mu+\eta^2\Delta_{12}+2n\eta\mu\delta-2\delta\eta n\mu-2n\mu\delta \eta^2\delta\\
 &=n\mu+\eta^2\Delta_{12}-n\mu\eta^2\delta^2\\
 &=n\mu \left [1-\eta^2 \left (\delta^2-\frac{\Delta_{12}}{n} \right ) \right].
 \end{align*}
\end{proof}

\begin{theorem}  \label{thm-wide} If we apply Algorithm \ref{wide-1} with  $\mathcal{N}_\infty^-(\frac{1}{2}) $, then the algorithm converges to an $\epsilon$-solution
in at most $O(n\ln(n)\ln \left ( \frac{1}{\epsilon} \right ))$ iterations for every $ \eta
=O(1)$.
\end{theorem}

\begin{proof} See Appendix \ref{proofs}.  \end{proof}

%%%%%%%%%%%%%%%%%%%%%%%%%%%%%%%%%%%%%%%%%%%%%%%%%%%%%%%%%%%%%%%%%%%%%%%%%%%%%%%%%%%%%%%%%%%%%%%%%%%%%%%%%%%%%%%%%%%%%%%
In the above algorithm, the value of $\eta$ is constant for all values of $j$. In the following, we show that if $\eta$ 
is allowed to take one of two constant values for each $j$ (one of the values being zero), we get a better iteration complexity bound. For each $j \in \{1,\ldots,n\}$, let us define
\begin{eqnarray} \label{m1}
[w(\eta)]_j:=\left \{ \begin{array} {ll} -v_j,  & \text{if} \ \ u_j > \frac{3}{4},  \\ -v_j+\eta(\delta v_j-v_j\ln (\frac{v^2_j}{\mu})), & \text{if} \ \ \frac{1}{2} \leq u_j \leq \frac{3}{4},
\end {array} \right.
\end{eqnarray} 
where $\eta:= \frac{1}{(\delta+\ln (2))}$. Now we have the following theorem:
\begin{theorem} \label{thm-wide-2}
If we apply the Algorithm \ref{wide-1} with $w(\eta)$ defined in \eqref{m1} to  $\mathcal{N}_\infty^-(\frac{1}{2}) $, the algorithm converges to an $\epsilon$-solution
in at most $O(n\ln \left ( \frac{1}{\epsilon} \right ))$ iterations.
\end{theorem}
\begin{proof}
See Appendix \ref{proofs}.
\end{proof}

%%%%%%%%%%%%%%%%%%%%%%%%%%%%%%%%%%%%%%%%%%%%%%%%%%%%%%%%%5
\section{Plane Search Algorithms} \label{plane-search}
%%%%%%%%%%%%%%%%%%%%%%%%%%%%%%%%%%%%%%%%%%%%%%%%%%%%%%%%%
In the previous section, we showed how to fix two parameters $\alpha$ and $\eta$ to achieve iteration complexity bounds. However, in practice we may consider performing a plane search to choose the best  $\alpha$ and $\eta$ in each iteration. Here, our goal is to choose a direction in our family of directions that gives the most reduction in the duality gap. As before, we have $w(\eta)=-v+\eta \left [\delta v-V\ln(\frac{Vv}{\mu}) \right ]$. For simplicity, in this section we drop parameter $\eta$ and write $w =w(\eta)$, so $w_p=P_{AD}w$, and $w_q=w-w_p$, where $P_{AD}$ is the projection operator onto the null space of $AD$. Our goal is to solve the following optimization problem. 
\begin{eqnarray}  \label{CS}
&\max& \alpha \nonumber \\
&\text{s.t.}& 0<\alpha<1,  \nonumber \\
&& \eta \geq 0, \nonumber \\
&& \frac{(w_p)_j(w_q)_j}{\mu}\alpha^2+\alpha \left (u_j\delta\eta-u_j\ln(u_j)\eta-u_j+\frac{1}{2} \right)+ \left (u_j-\frac{1}{2} \right) \geq 0, \ \ \ \forall j \in \{1, \ldots,n\}.
\end{eqnarray}
In the above optimization problem, the objective function is linear and the main constraints are quadratic. Let us define 
$$
t_p:=P_{AD} \left (\delta v - V \ln \left (\frac{Vv}{\mu} \right) \right), \ \ t_q:= \left (\delta v - V \ln \left (\frac{Vv}{\mu} \right) \right) - t_p, \ \ v_p:= P_{AD}(-v), \ \  v_q:=-v-v_p. 
$$
By these definitions, the quadratic inequalities in formulation \eqref{CS}  become
\begin{eqnarray*}
&&a_j\eta^2\alpha^2 + b_j \eta \alpha + c_j \eta \alpha^2  + d_j (1-\alpha)  + e_j \alpha^2  \geq 0, \ \ \text{where}   \\
&&a_j:=\frac{(t_p)_j(t_q)_j}{\mu}, \ b_j:= \frac{(v_p)_j(t_p)_j+(v_q)_j(t_p)_j}{\mu}, \ c_j:=u_j \delta - u_j \ln (u_j), \\
&&e_j := \frac{(v_p)_j(v_q)_j}{\mu}, \ d_j := u_j - \frac 12.
\end{eqnarray*}
 In this section, we propose two algorithms, an exact one and a heuristic one, to solve the two-variable optimization problem \eqref{CS}.  
 \subsection{Exact plane search algorithm}
We define a new variable $z:= \alpha \eta$. Then, the quadratic form can be written as:
$$
g_j(z,\alpha) := a_jz^2 + b_jz + c_j z \alpha  + d_j (1-\alpha)  + e_j \alpha^2, \ \ \  \forall j \in \{1, \ldots,n\}.
$$
We are optimizing in the plane of $\alpha$ and $z$, actually working in the one-sided strip in $\mathbb R^2$, defined by $0  \leq  \alpha \leq 1$ and $z \geq 0$. The following proposition establishes that
it suffices to check $O(n^2)$ points to find the optimal solution:
\begin{proposition}
Let $(\alpha^*,\eta^*)$ be an optimal solution of \eqref{CS}. Then, one of the following is true:
\begin{enumerate}
\item{$\alpha^*=1$;}
\item{there exists $z^* \geq 0$ such that $(z^*,\alpha^*)$ is a solution of system $\left (\begin{array} {c} g_j(z,\alpha) \\ g_i(z,\alpha) \end{array} \right ) = \left (\begin{array} {c} 0 \\ 0 \end{array} \right )$ for some pair $i,j \in \{1, \ldots,n\}$; }
\item{$\alpha^*$ is a solution of $\Delta_j(\alpha) := (b_j+\alpha c_j)^2-4a_j(d_j (1-\alpha)  + e_j \alpha^2) = 0$, $j \in \{1, \ldots,n\}$, where $\Delta_j(\alpha)$ is the discriminant of $g_j(z,\alpha)$ with respect to $z$. }
\end{enumerate} 
\end{proposition}
\begin{proof}
Assume that $(\alpha^* \neq 1, \eta^*)$ is a solution to \eqref{CS}, and $z^*:=\alpha^* \eta^*$. Therefore, we have $g_j(z^*,\alpha^*) \geq 0$, $ \forall j \in \{1, \ldots,n\}$. By continuity, we must have $g_j(z^*,\alpha^*) = 0$ for at least one $j$, and because $z^*$ is real, we have $\Delta_j(\alpha^*) \geq 0$. If  $\Delta_j(\alpha^*) = 0$, then condition (3) is satisfied, otherwise, by continuity, we can increase  $\alpha$ so that $\Delta_j$ remains positive. In this case, if there does not exist another $i \in  \{1, \ldots,n\}$ that $g_i(z^*,\alpha^*) = 0$, continuity gives us another point $( \bar \alpha, \bar \eta)$ that is feasible to \eqref{CS} and $ \bar \alpha >  \alpha^*$, which is a contradiction. Hence, condition (2) must hold. 
\end{proof}
The above proposition tells us that to find a solution for \eqref{CS}, it suffices to check $O(n^2)$ values for $\alpha$. For calculating each of these values, we find the roots of a quartic equation. 

\subsection{Heuristic plane search algorithm}
The idea of the heuristic algorithm is that we start with $\alpha=1$ and see if there exists $\eta$ such that $(\eta,\alpha)$ is feasible for \eqref{CS}. If not, we keep reducing $\alpha$ and 
repeat this process. We can reduce $\alpha$ by a small amount (for example $0.01$) if $\alpha$ is close to 1 (for example $\alpha \geq 0.95$), and by a larger amount (for example $0.05$) otherwise. This tries to favor the larger $\alpha$ values over the smaller ones.

%\begin{algorithm}
%Input $(A, x^{(0)}, s^{(0)},  b, c,\beta, \epsilon)$, where $A, x^{(0)}, s^{(0)},
%b, c$ are defined in the \textup{(HLP)} formulation in section 2,\\
% $\epsilon$ is the desired tolerance, $\beta$ determines the wide neighbourhood we
% set.\\
%while $x^{\top}s>\epsilon$ \+\\
%     $\alpha : = 1$\+\\
%       calculate $\delta, D, v, w(\eta), w(\eta)_p, w(\eta)_q$ as defined before,\\
%       and check if there exists $\eta$ such that $(\eta,\alpha)$ is feasible for (CS). \\
%       If such $\eta$ exits return $(\eta,\alpha)$, otherwise reduce $\alpha$ as follows: \\
%       If $\alpha > 0.95$, $\alpha:= \alpha - 0.01$, else $\alpha:= \alpha - 0.05$. \-\-\\
%   repeat\\
%\end{algorithm}
The difficult part is checking if there exists $\eta$ for the current $\alpha$ in the algorithm. 
To do that, we need to verify if there exists a positive $\eta$ which satisfies the $n$ inequalities in the constraints (3) of \eqref{CS}. Each constraint is a quadratic form in $\eta$ and can induce a feasible interval for $\eta$. If the intersection of all the intervals corresponding to these $n$ inequality constraints is not empty, we then find the $\eta$ corresponding to a step length $\alpha$. We use the following procedure to determine the feasible interval of $\eta$ for a given step length $\alpha$.

Assume that we fix $\alpha$. For each quadratic constraint of \eqref{CS}, we can solve for $\eta$ and find the feasible interval. One form is the union of two open intervals, i.e., $(-\infty, r_1(j)]$ and
$[r_2(j),\infty)$, denote the indexes in this class as $K_1$. Another is the convex interval $[r_3(j), r_4(j)]$, denote  the indexes in this class by $K_2$. It is easy to find the intersection of the convex intervals:
$$
[t_1,t_2]:=[\max_{j\in K_2} r_3(j),\min_{j\in K_2} r_3(j)].
$$
Now we have to intersect $[t_1,t_2]$ with the intervals in class $K_1$. First we handle the intervals that  $[t_1,t_2]$ intersects only one of $(-\infty, r_1(j)]$ and $[r_2(j),\infty)$; in that case we can update \mbox
{$[t_1,t_2] \leftarrow [t_1,r_1(j)]$} or $[t_1,t_2] \leftarrow [r_2(j),t_2]$ for each of these intervals. At the end of this step, we can assume that for the rest of the intervals in $K_2$ (we denote them by $\bar K_2$), $[t_1,t_2]$ intersects both $(-\infty, r_1(j)]$ and $[r_2(j),\infty)$. Then, we can define two intervals:
$$
[t_1,t_3:=\min_{j\in \bar K_2}(r_1(j))], \ \ [t_4:=\max_{j\in \bar K_2}(r_2(j))],t_2]
$$ 
If one of these intervals is non-empty, then there exists $\eta$ such that $(\eta,\alpha)$ is feasible for \eqref{CS}, and we return $\alpha$. For a more detailed introduction to this heuristic see \cite{Shen-luo}.

To evaluate the performance of our heuristic algorithm, note that the set of feasible points $(\alpha,
\eta)$ of \eqref{CS} in $\mathbb R^2$ is not necessarily a connected region. We can think of it as the union of many connected components. In our heuristic algorithm, we check a few discrete values of $\alpha=\bar \alpha$. However, for each value we check, we can precisely decide if there exists a feasible $\eta$ for that value of $\alpha$. If one of the lines $\alpha=\bar \alpha$ intersects a component of feasible region that contains a point with maximum $\alpha$, then our heuristic algorithm returns an $\alpha$ that is close the optimal value. However, if none of the lines $\alpha=\bar \alpha$ that we check for large values of $\bar \alpha$ intersects the right component, the heuristic algorithm may return a very bad estimate of the optimal value. In the next section, we observe that (see Figures \ref{Fig-degen2_m0}--\ref{Fig-ship08s_m1}) our heuristic algorithm in the worst-case may return values for $\alpha$ very close to zero while the optimal value is close to 1. 

%%%%%%%%%%%%%%%%%%%%%%%%%%%%%%%%%%%%%%%%%%%%%%%%%%%%%%%%%%%%5
\section{Computational Experiments with the Entropic Search
Direction Family} \label{num}
%%%%%%%%%%%%%%%%%%%%%%%%%%%%%%%%%%%%%%%%%%%%%%%%%%%%%%%%%%%%

We performed some computational experiments using the software MATLAB R2014a, on a 48-core AMD Opteron 6176 machine with 256GB of memory. 
 The test LP problems are well-known among those in the
problem set of NETLIB \cite{netlib}. 

We implemented Algorithm \ref{wide-1} for a fixed value of $\eta$ and then ran it for each fixed $\eta \in \{1,2,3,4\}$. We also implemented Algorithm \ref{wide-1} with $\eta$ being calculated using the exact and heuristic plane search algorithms. $\beta=1/2$ was set for the algorithm, therefore our results are for the wide neighborhood $\mathcal{N}_\infty^-(1/2)$. We used homogeneous self-dual embedding for the LP problems as shown in Appendix \ref{App2}. The initial
feasible solution is  $y^{(0)}:=0$, $x^{(0)}:=e$, $s^{(0)}:=e$, $\theta:=1$,
$t:=1$ and $\kappa:=1$. In the statements of Algorithms \ref{PC} and \ref{wide-1}, we used the stopping criterion $x^Ts \leq \epsilon$, which is an abstract criterion assuming exact arithmetic computation. In practice, we may encounter numerical inaccuracies and we need to take that into account for our stopping criterion. 
We used the stopping criterion proposed and studied in \cite{behavior}, which is very closely related to the stopping criterion in SeDuMi \cite{sturm2002}. Let us define $(\bar x, \bar y, \bar s):=(\frac{x}{\tau}, \frac{y}{\tau}, \frac{s}{\tau})$, and their residuals:
\begin{eqnarray*}
r_p &:=& b-A \bar x, \nonumber \\
r_d &:=& A^ \top \bar y + \bar z -c, \nonumber \\
r_g &:=& c^ \top \bar x - b^ \top \bar y. 
\end{eqnarray*}
The following stopping criterion for general convex optimization problems using homogeneous self-dual embedding was proposed in \cite{behavior}:
\begin{eqnarray*}
2 \frac{\|r_p\|_\infty}{1+ \|b\|_\infty}+2 \frac{\|r_d\|_\infty}{1+ \|c\|_\infty}+\frac{(r_g)^+}{\max\left \{|c^ \top \bar x|,|b^ \top \bar y|,1 \right \}} \leq r_{max}.
\end{eqnarray*}
In our algorithm, we used the above stopping criterion for $r_{max} := 10^{-9}$. 

%In order to avoid the accumulated error in the equality constraints,
%after each iteration we let:
%$$\overline{b}=\overline{b}-\chi_1/\theta, \ \ 
%\overline{c}=\overline{c}+\chi_2/\theta, \ \
%\overline{z}=\overline{z}-\chi_3/\theta, $$
%where $\chi_i$ is the error i.e., 
%$$\chi_1=Ax-bt+\overline{b}\theta, \ \
%\chi_2=A^{\top}y+ct-\overline{c}\theta-s,\ \
%\chi_3=b^{\top}y-c^{\top}x+\overline{z}\theta-\kappa.$$
%By u                                        sing this method, we can dump errors into the column that is
%guaranteed to be a nonbasic column.

Table 1 shows the number of iterations for each problem. The first four columns show
the number of iterations of
Algorithm \ref{wide-1} with a fixed value of $\eta \in \{1,2,3,4 \}$. Let use define $\tilde \eta$ and $\eta^*$ as the $\eta$ found at each iteration of the plane search algorithm using the heuristic and exact plane search algorithms, respectively. The fifth and sixth columns of the table are the number of iterations when 
we perform a plane search, using the heuristic plane search and exact plane search algorithms, respectively. The problems in the table are sorted based on the value of $\eta \in \{1,\ldots,4\}$ that gives the smallest number of iterations. For each $\eta$, the problems are sorted alphabetically.

\newpage
\begin{center}
Table 1: The number of iterations of Algorithm \ref{wide-1}. 
\end{center}
\begin{eqnarray}\nonumber
\def\arraystretch{1.3}
\begin{array}{|l r c||c |c |c |c||c|| c |}
\hline
{\rm NETLIB-Name} & \rm Dimensions& \rm Nonzeros
 &\eta=1 & \eta=2 & \eta=3 & \eta=4  &\tilde \eta  & \eta^*\\
\hline \hline
afiro&28*32&88&\mathbf{31}& 37 & 45 & 55 & 19 & 18 \\
\hline
beaconfd &174*262 &3170&   \mathbf{31} & 37 & 46 &56 &22 & 21  \\
\hline 
blend &75*83 &522&    \mathbf{ 35} & 37 &42 &48 &23 &19 \\
\hline 
grow7 &141*301 &2633&\mathbf{47} & 47& 55 &65 & 36 & 31 \\
\hline
grow15&301*645 & 5665 &\mathbf{42}& 47&54&66& 37 & 32 \\
\hline
sc105&106*103& 281 &\mathbf{28} & 36&43 &54&23 & 19\\
\hline
sc205 &206*203 &552&\mathbf{28} & 37&41 &53 &25 & 21\\
\hline
sc50a &51*48 &131&\mathbf{28} & 35& 43 &51 & 21 & 18 \\
\hline
sc50b &51*48 &119&\mathbf{26} & 33& 43 &50&19 & 16\\
\hline
scagr7&130*140&533&\mathbf{34}&39&47 &55&  28 & 25 \\
\hline
scsd1&78*760&3148& \mathbf{31} &35&43 &52&26 & 18 \\
\hline
scsd8&398*2750&11334&\mathbf{32} &34& 42 &48& 25  & 19 \\
\hline
share2b&97*79&730&\mathbf{35 }& 40 &45  &54&  25 & 23\\
\hline
adlittle&57*97&465&\mathbf{40} &\mathbf{40}&49 &56 &27 & 24 \\
\hline
kb2&44*41&291&     \mathbf{43}& \mathbf{43} &49 &58 & 32 & 30 \\
\hline
agg&489*163&2541&    58 &\mathbf{50}&58 &68& 46 & 39\\
\hline
agg2&517*302&4515&61&\mathbf{51}&56 &66& 40 & 35\\
\hline
agg3&517*302&4531&  70  &\mathbf{55}&61 &68 & 42 & 37  \\
\hline
boeing2& 167*143& 1339 & 72 & \mathbf{53} & 54 & 62 & 39 & 37 \\
\hline
brandy&221*249&21506&   72 &\mathbf{56}& 56 &57 & 43 & 36 \\
\hline
capri&272*353&1786&  60  &\mathbf{49}&53 &57 & 39 & 37  \\
\hline
degen2&445*534&4449& 41 & \mathbf{38} & 42 &48& 37 & 24 \\
\hline
\end{array}
\end{eqnarray}

\newpage

\begin{eqnarray}\nonumber
\def\arraystretch{1.3}
\begin{array}{|l r c||c |c |c |c||c|| c |}
\hline
{\rm NETLIB-Name} & \rm Dimensions& \rm Nonzeros
 &\eta=1 & \eta=2 & \eta=3 & \eta=4  &\tilde \eta  & \eta^*\\
\hline \hline
degen3&1504*1818 &26230     &40 &\mathbf{38}&40  & 45& 30 & 26 \\
\hline
fit1d & 25*1026 & 14430 & 68 & \mathbf{53} & 64 & 64 & 38 & 37 \\
\hline
forplan & 162*421 & 4916 &     108 &  \mathbf{77 } & 79 & 119 & 60 & 50 \\
\hline
ganges&1310*1681&70216&   52 &\mathbf{48}& 54 &63 & 41 & 38 \\
\hline
gfrd-pnc&617*1092&3467&50 &\mathbf{47}&53 &63& 35 & 30 \\
\hline
grow22&441*946 & 8318 &51 & \mathbf{50}&55&67& 38 & 34\\
\hline
lotfi&154*308& 1086&  54 &\mathbf{46}&50& 58& 35 & 32\\
\hline
scagr25&472*500&2029& 43 &\mathbf{42}&51 &58&  33 & 31 \\
\hline
scsd6&148*1350 & 5666 &35 & \mathbf{36}&44&51& 25 & 22 \\
\hline
sctap2&1091*1880&8124&  56 & \mathbf{48} & 49& 52& 34 & 23 \\
\hline
ship04s&403*1458 &5910&  53 &\mathbf{45} &52&54&34 & 30\\
\hline
ship04l&403*2118 &8450&    53 &\mathbf{49} &56&58& 35 & 29\\
\hline
stocfor1 & 118*111 & 474 & 51 & \mathbf{48} & 55 & 61 & 30 & 26 \\
\hline
wood1p&245*2594&70216&120 &\mathbf{62}& 104 &65 & 53 & 52 \\
\hline
fit1p& 628*1677 & 10894 &    63 &\mathbf{54 } & \mathbf{54 } & 59 & 37 & 35 \\
\hline
bandm & 305*472 & 2659 & 67 & 54 &  \mathbf{51} & 58 &  40 & 35 \\
\hline
boeing1&351*384 & 3865 & 91& 68 & \mathbf{65} & 68 & 51 & 46 \\
\hline
e226&224*282 &2767&  66&53& \mathbf{52} & 55 &40 & 36 \\
\hline
israel& 175*142& 2358&96 &69 & \mathbf{67}  &73 &46 & 40\\
\hline 
d6cube &404*6184&37704&   76 & 58& \mathbf{55} &  56 & 39  & 32 \\
\hline
modszk1& 686*1622 & 3170 &     110 & 87 & \mathbf{77 } & 82 & 75 & 53 \\
\hline
scfxm1&331*457&2612&    123&81 &\mathbf{73} & 74 &53 & 42\\
\hline
\end{array}
\end{eqnarray}

\newpage

\begin{eqnarray}\nonumber
\def\arraystretch{1.3}
\begin{array}{|l r c||c |c |c |c||c|| c |}
\hline
{\rm NETLIB-Name} & \rm Dimensions& \rm Nonzeros
 &\eta=1 & \eta=2 & \eta=3 & \eta=4  &\tilde \eta  & \eta^*\\
 \hline \hline
 scrs8 & 491*1169 & 4029 & 143 & 105 &  \mathbf{95} & 100 & 50 & 44 \\
\hline
sctap3&1481*2480 & 10734 &64 & 53& \mathbf{53} & 57& 37 & 25 \\
\hline
ship08s&779*2387&9501&    82 & 62 &\mathbf{62}&63& 54 & 31 \\
\hline
vtp-base &199*203 &914& 117 & 84 & \mathbf{76} & 77 &50 &36\\
\hline
scfxm3&991*1371&7846&   143& 98 & \mathbf{84}&\mathbf{84} &  70 & 47 \\
\hline
25fv45 & 822*1571 & 11127 & 160 & 111 & 91 &  \mathbf{79} & 77 & 57 \\
\hline
bnl1&644*1175&6129&  183 &122& 103 &\mathbf{95} &87& 64 \\
\hline
bnl2&2325*3489&16124&    197 &137& 110 & \mathbf{100}& 78 & 71 \\
\hline
czprob & 930*3523 & 14173 & 236 & 156 & 129 &  \mathbf{119} & 70 & 61 \\
\hline
etamacro & 401*688  & 2489 & 218 & 134 & 109 & \mathbf{98} & 66 & 59 \\
\hline
pilot4 & 411*1000 & 5145  & 150 &107  &  91 & \mathbf{88}  & 87 & 77 \\
\hline
pilot-we & 723*2789 & 9218 & 234 & 164 & 139 & \mathbf{125} & 119 & 118 \\
\hline
perold & 626 *1376 & 6026 & 190 & 123 & 101 & \mathbf{95} &  94 & 93 \\
\hline
scfxm2&661*914&5229& 146 &97 & 83 &  \mathbf{81} & 69 & 47\\
\hline
sctap1&301*480&2052&  119&81& 75&\mathbf{67}&37 &35\\
\hline
seba & 516*1028 & 4874 & 170 & 120 & 103 & \mathbf{94} & 78 & 50 \\
\hline
share1b&118*225&1182&  128 &84 &  74 & \mathbf{73}  & 58 &51 \\
\hline
ship12l & 1152*5437 & 21597 & 272 & 164 & 133& \mathbf{120} & 75 &  50 \\
\hline
ship12s&1152*2763&10941&    218 &134 & 106 &\mathbf{97} &74 &42 \\
\hline
stocfor2 &2158*2031 &9492& 129 & 95& 83 & \mathbf{82}& 63 & 56 \\
\hline
standata & 360*1075 & 3038 & 108 & 71 & 67 & \mathbf{66} & 33  & 27 \\
\hline
standmps & 468*1075 & 3686 & 125 & 85 & 73 & \mathbf{70} & 63 & 35 \\
\hline
\end{array}
\end{eqnarray}

\newpage

As we mentioned above, our family of search directions is a common generalization of the search direction in \cite{TT97} that uses $\eta=1$ and the search directions in \cite{ZL2003} and \cite{pan} that use $\eta=\frac 1 \sigma$ with $\sigma \in
(0.5,1)$ and $\sigma < \min \left \{1, \ln \left ( \frac{1}{1-\beta} \right ) \right\}$, so $1 \leq \eta \leq 2$. As we observe from Table 1, our generalization to consider using larger values of $\eta$ is justified. Among the problems solved and among the fixed values for $\eta \in \{1,2,3,4\}$, $\eta=1$ had the smallest iteration count for 15 problems, $\eta=2$ won for 24 problems, $\eta=3$ won for 13 problems, and $\eta=4$ had the smallest iteration count for 18 problems (ties counted as wins for both winning $\eta$'s). Table 1 also shows that using plane search algorithms can be crucial in reducing the number of iterations in addition to making the behaviour of the underlying algorithms more robust; as (1) for most of the problems, there is a large gap between the number of iterations of the plane search and the best constant $\eta$ algorithms, and (2) we do not know which $\eta$ is the best one before solving the problem. 

The exact plane search algorithm  gives a lower bound for our heuristic plane search algorithm. As we observe from Table 1, for most of the problems, exact and heuristic plane search algorithms have similar performances in terms of the number of iterations. In Figures \ref{Fig-beaconfd}--\ref{Fig-ship08s}, we plot the value of $\eta$ at each iteration for four of the problems of NETLIB, for both exact and heuristic plane search algorithms. For \emph{beaconfd} and \emph{capri} the performances are close and for \emph{degen2} and \emph{ship08s} there is a large gap. An interesting point is that the plane search algorithms sometimes lead to values of $\eta$ as large as 10 or 20 as can be seen in Figures \ref{Fig-degen2} and \ref{Fig-ship08s}.

Figures \ref{Fig-degen2_m0}--\ref{Fig-ship08s_m1} provide a more reasonable comparison between the exact and heuristic plane search algorithms for problems  \emph{degen2} and \emph{ship08s}. In Figure \ref{Fig-degen2_m0} (for \emph{degen2}) and Figure \ref{Fig-ship08s_m0} (for \emph{ship08s}), we plot the values of $\eta$ and $\alpha$ for the heuristic algorithm, as well as the corresponding values that would have been computed by the exact algorithm at each iteration (for the same current iterates $(x^{(k)},s^{(k)})$). In Figure \ref{Fig-degen2_m1} (for \emph{degen2}) and Figure \ref{Fig-ship08s_m1} (for \emph{ship08s}), we plot the values of $\eta$ and $\alpha$ for the exact algorithm, as well as the corresponding values that would have been given by the heuristic algorithm at each iteration. We observe from the figures that when the optimal value of $\alpha$ is close to 1 or 0, the heuristic algorithm cannot keep up with the exact algorithm. 

A conclusion of the above discussion is that utilization of plane search algorithms improves the number of iterations significantly. If the plane search algorithm is fast enough, then we can also improve the running time. Our heuristic plane search algorithm is much faster than the exact one. For the exact plane search algorithm, we solve $O(n^2)$ quartic equations, and in each iteration of the primal-dual algorithm, we perform $O(n^3)$ operations. Therefore, if we can speed up our exact plane search algorithm, this would have a potential impact on practical performance of algorithms in this paper as well as some other related algorithms. Note that our main focus in these preliminary computational experiments is on the number of iterations. To speed up the plane search algorithms, one may even use tools from computational geometry, analogous to those used for solving two-dimensional (or $O(1)$-dimensional) LP problems with $n$ constraints in $O(n)$ time (see \cite{megiddo1982linear}, \cite{dyer1992class}, and the book \cite{edelsbrunner1987algorithms}).

%%%%%%%%%%%%%%%%%%%%%%%%%%%%%%%%%%%%%%%%%%%
\section{Conclusion} \label{con}

In this paper, we introduced a family of search directions parameterized by $\eta$. We proved that if we use our search direction with $\eta=1$ in the predictor step of standard predictor-corrector algorithm, we can achieve the current best iteration complexity bound. Then, we focused on the wide neighborhoods, and after the derivation of some theoretical results, we studied the practical performance of our family of search directions. To find the best search direction in our family,  which gives the largest decrease in the duality gap, we proposed a heuristic plane search algorithm as well as an exact one. Our experimental results showed that using plane search algorithms improves the performance of the primal-dual algorithm significantly in terms of the number of iterations. Although our heuristic algorithm works efficiently, there is more room here to work on other heuristic plane search algorithms or improving the practical performance of the exact one, such that we also obtain a significant improvement in the overall running time of the primal-dual algorithm. 

The idea of using a plane search in each iteration of a primal-dual algorithm has been used by many other researchers. For example, relatively recently, Ai and Zhang \cite{ai-zhang} defined a new wide neighborhood (which contains the conventional wide neighborhood for suitable choices of parameter values) and introduced a new search direction
by decomposing the RHS vector of \eqref{KKT-2} into positive and negative parts and performing a plane search to find the step size for each vector. By this approach, they obtained the current best iteration complexity bound for their wide neighborhood. Their approach together with ours inspires the following question: are there other efficient decompositions which in combination with a plane search, give good theoretical as well as computational performances in the wide neighborhoods of the central path? This is an interesting question left for future work. 

%%%%%%%%%%%%%%%%%%%%%%%%%%%%%%%%%%%%%%%%%%%

\renewcommand{\baselinestretch}{1}
\bibliographystyle{plain}
\bibliography{References}

\newpage 
\begin{figure} [ht]
\begin{center}
    \includegraphics  [scale=.5]{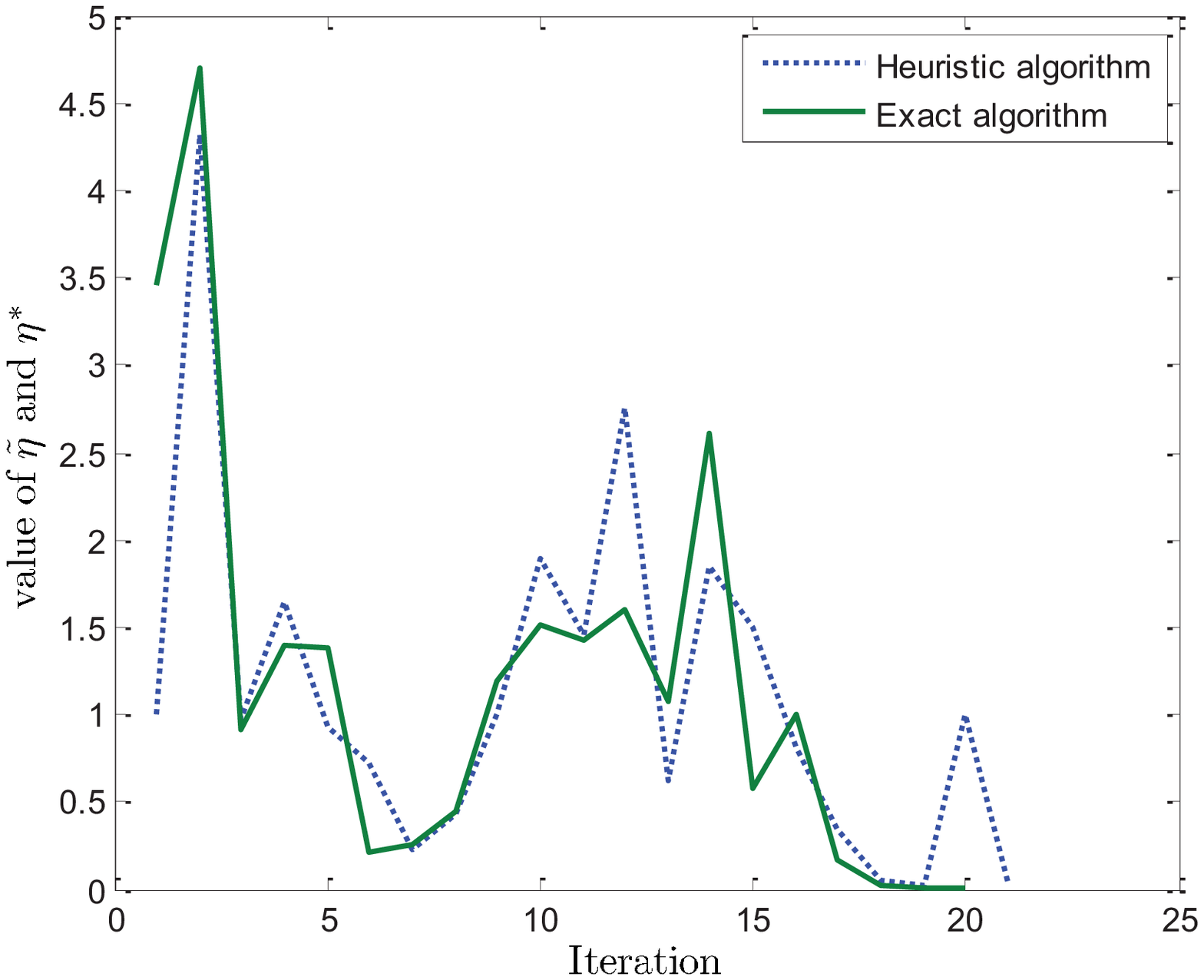}
\end{center}
\caption{\small Values of $\tilde \eta$ (for the heuristic algorithm) and $\eta^*$ (for the exact algorithm) in each iteration for problem \emph{beaconfd}. \label{Fig-beaconfd}}
\end{figure}

\begin{figure} [ht]
\begin{center}
    \includegraphics  [scale=.5]{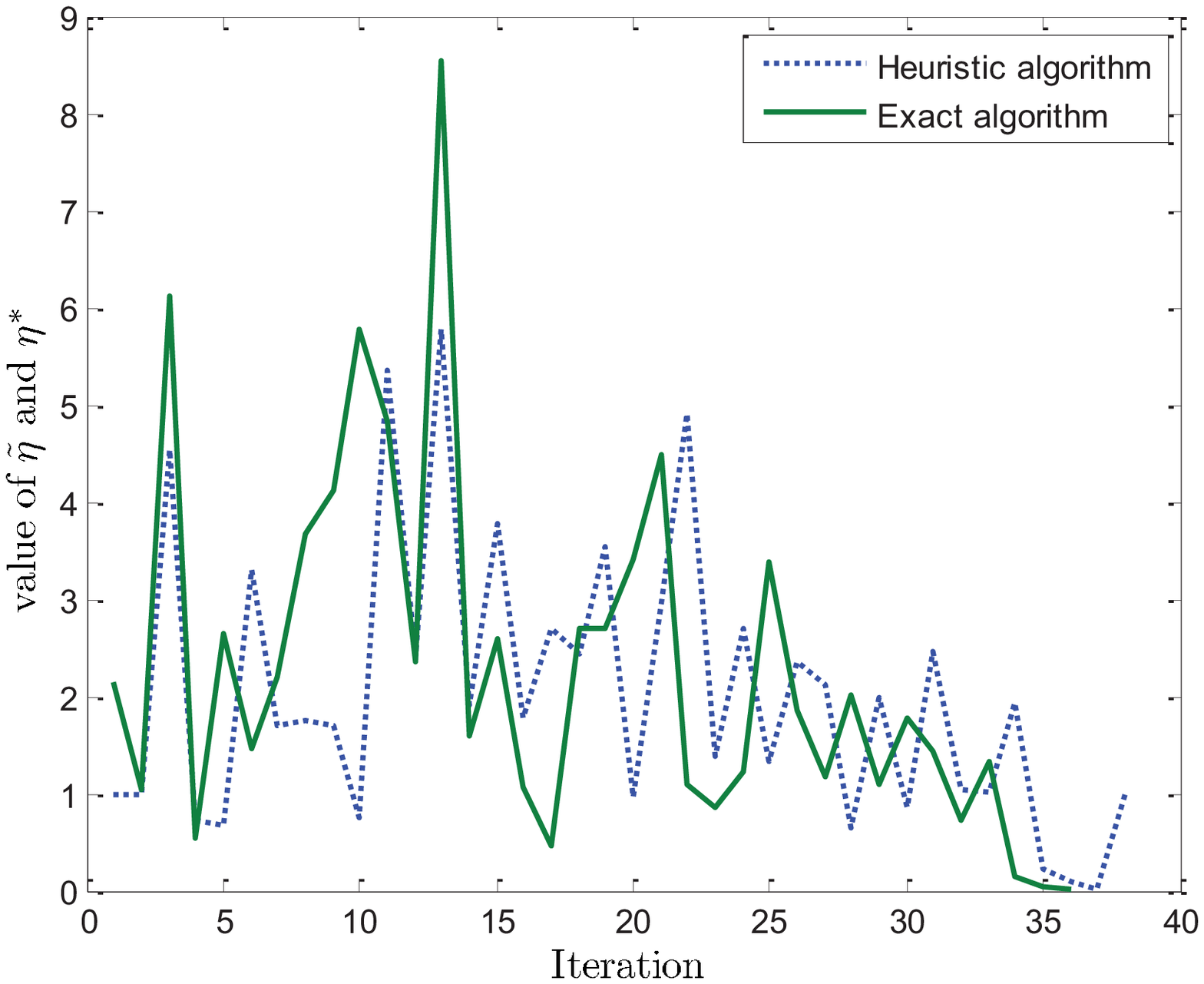}
\end{center}
\caption{\small Values of $\tilde \eta$ (for the heuristic algorithm) and $\eta^*$ (for the exact algorithm) in each iteration for problem \emph{capri}. \label{Fig-capri}}
\end{figure}

\newpage

\begin{figure} [ht]
\begin{center}
    \includegraphics  [scale=.5]{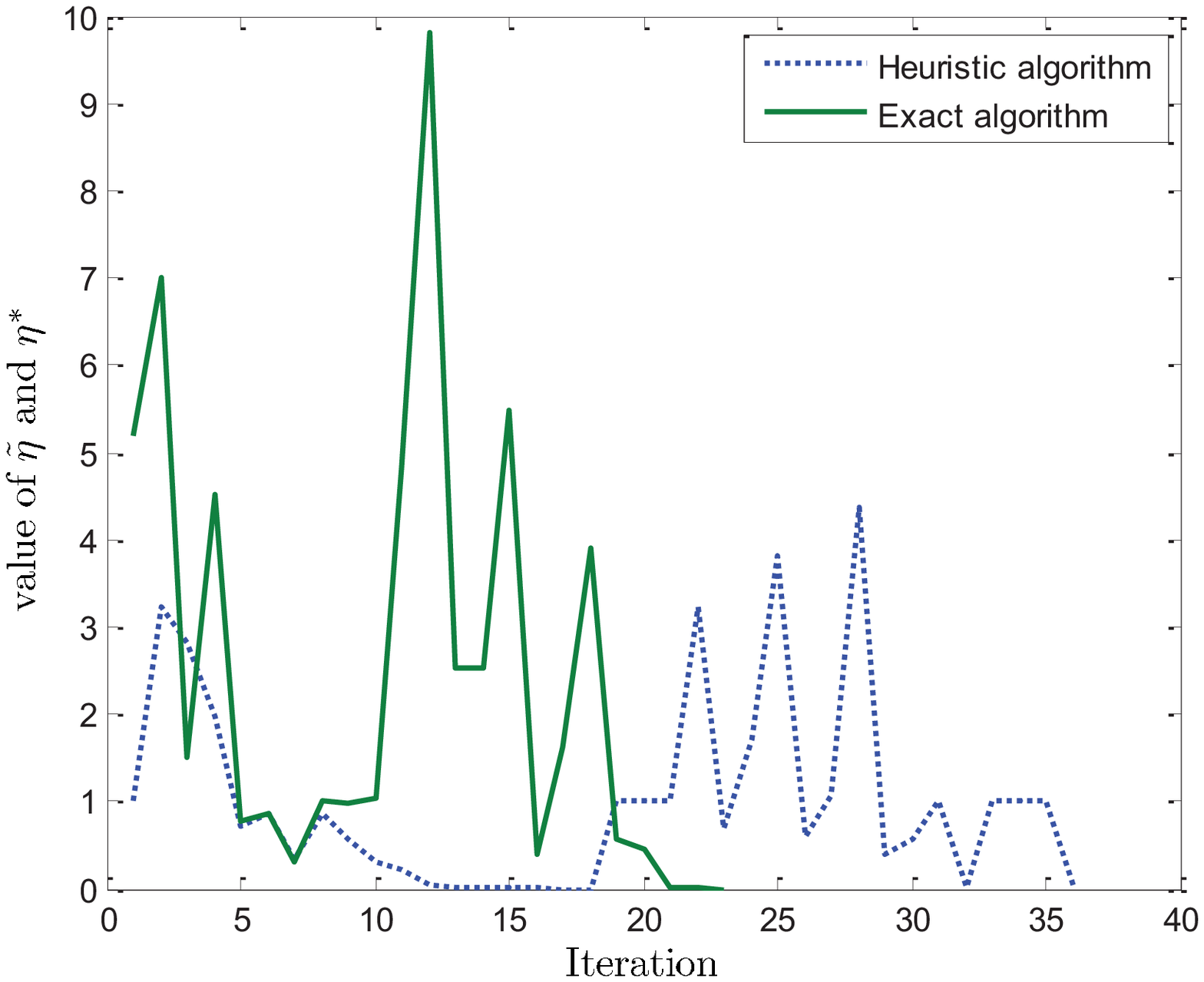}
\end{center}
\caption{\small Values of $\tilde \eta$ (for the heuristic algorithm) and $\eta^*$ (for the exact algorithm) in each iteration for problem \emph{degen2}. \label{Fig-degen2}}
\end{figure}

\begin{figure} [ht]
\begin{center}
    \includegraphics  [scale=.5]{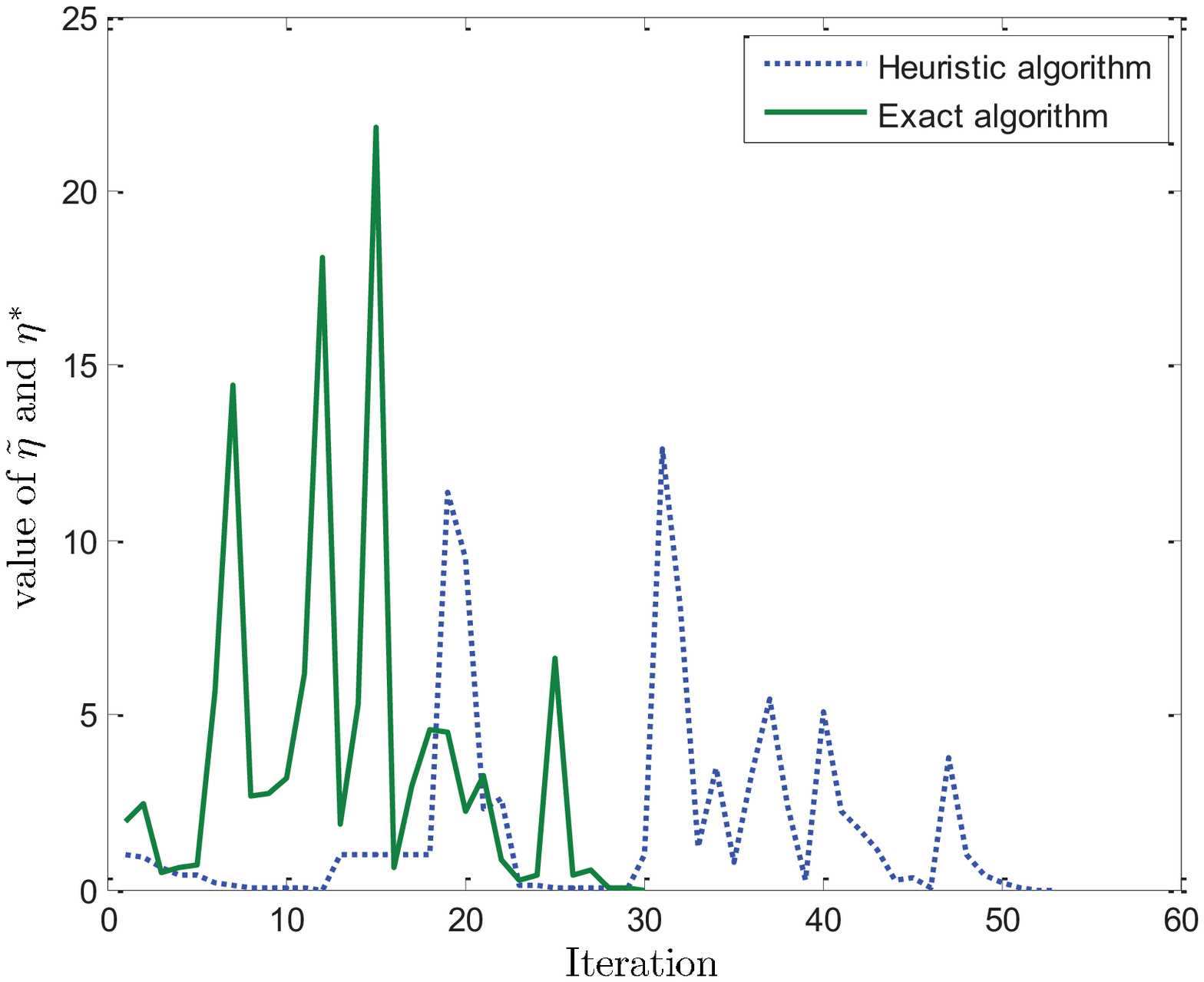}
\end{center}
\caption{\small Values of $\tilde \eta$ (for the heuristic algorithm) and $\eta^*$ (for the exact algorithm) in each iteration for problem \emph{ship08s}. \label{Fig-ship08s}}
\end{figure}

\newpage

\begin{figure} [ht]
\begin{center}
    \includegraphics  [scale=.45]{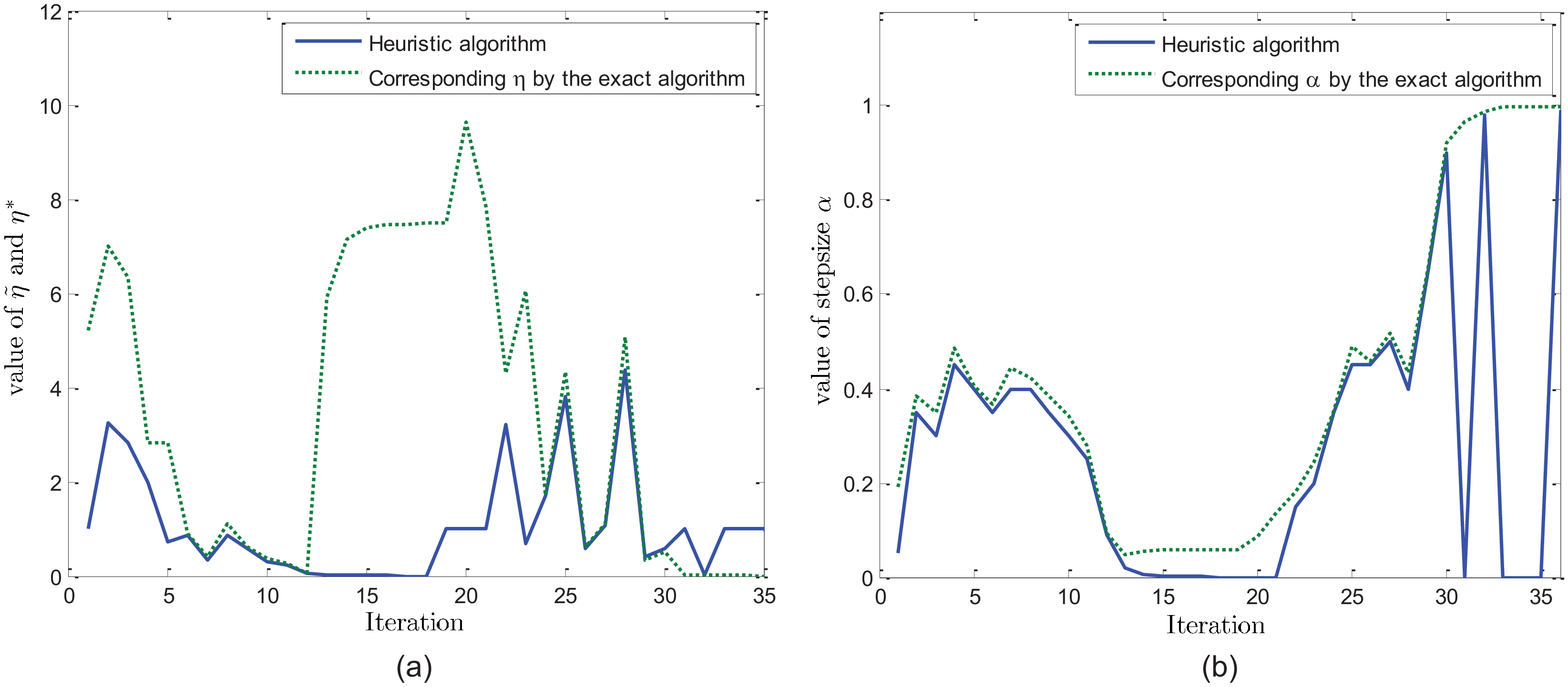}
\end{center}
\caption{\small Values of (a) $\eta$ (b) $\alpha$ for the heuristic algorithm, and the corresponding values calculated by the exact algorithm at each iteration of it, for problem \emph{degen2}. \label{Fig-degen2_m0}}
\end{figure}

\begin{figure} [ht]
\begin{center}
    \includegraphics  [scale=.45]{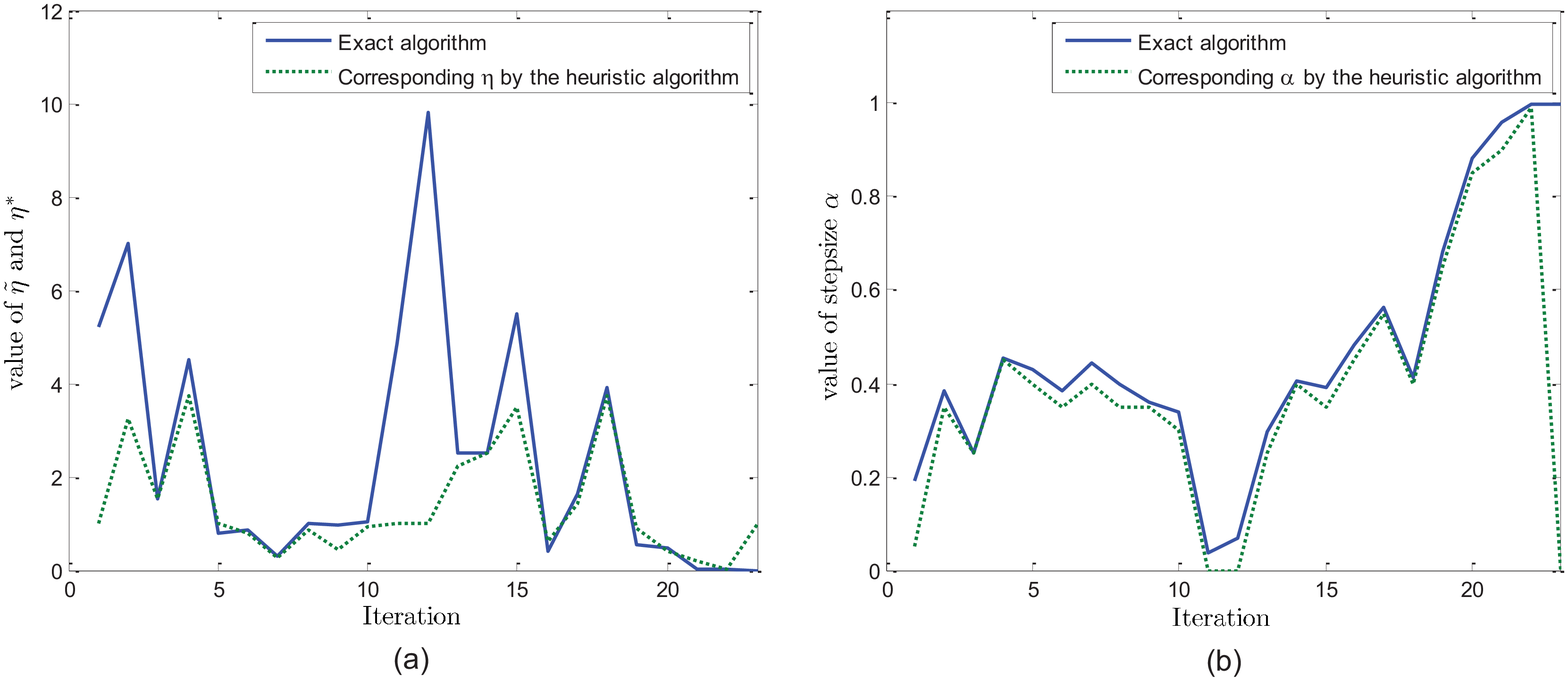}
\end{center}
\caption{\small Values of (a) $\eta$ (b) $\alpha$ for the exact algorithm, and the corresponding values calculated by the heuristic algorithm at each iteration of it, for problem \emph{degen2}. \label{Fig-degen2_m1}}
\end{figure}

\newpage

\begin{figure} [ht]
\begin{center}
    \includegraphics  [scale=.45]{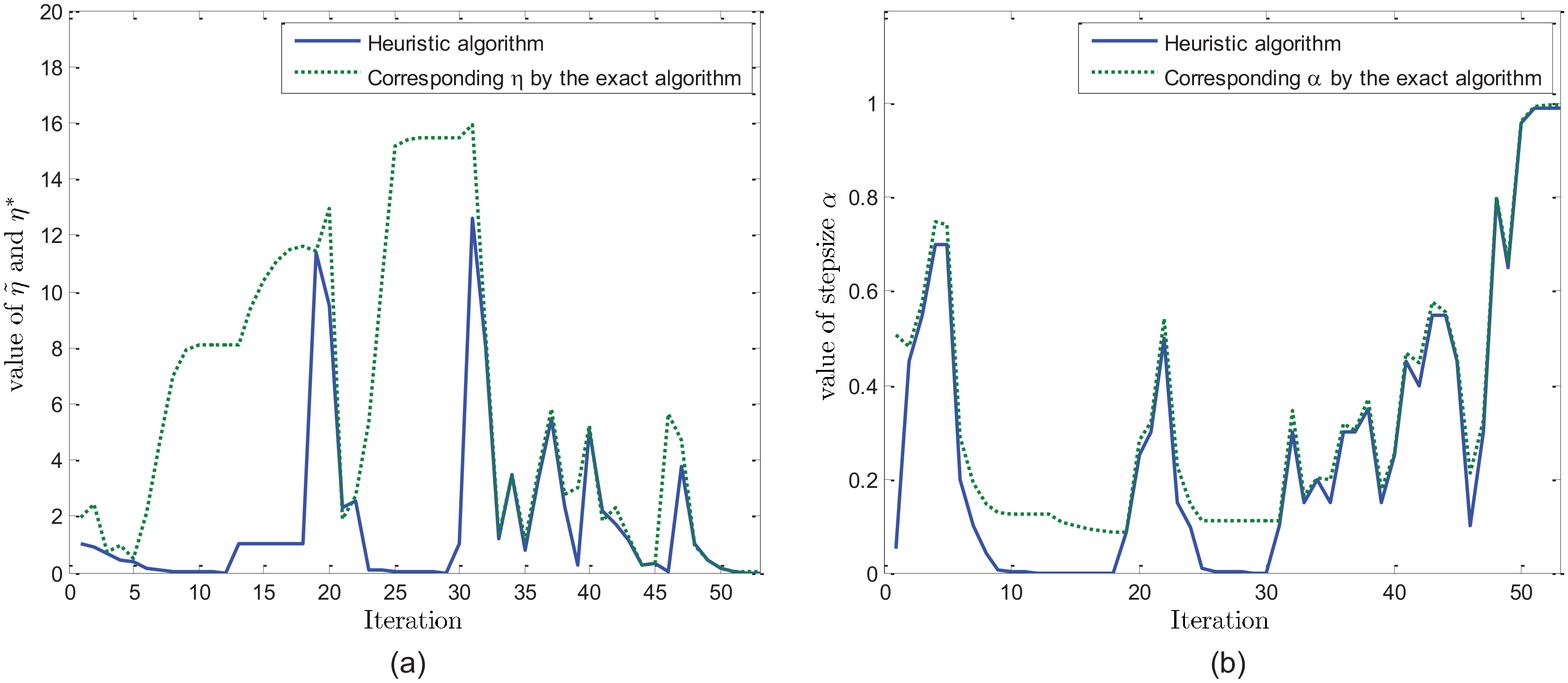}
\end{center}
\caption{\small Values of (a) $\eta$ (b) $\alpha$ for the heuristic algorithm, and the corresponding values calculated by the exact algorithm at each iteration of it, for problem \emph{ship08s}. \label{Fig-ship08s_m0}}
\end{figure}

\begin{figure} [ht]
\begin{center}
    \includegraphics  [scale=.45]{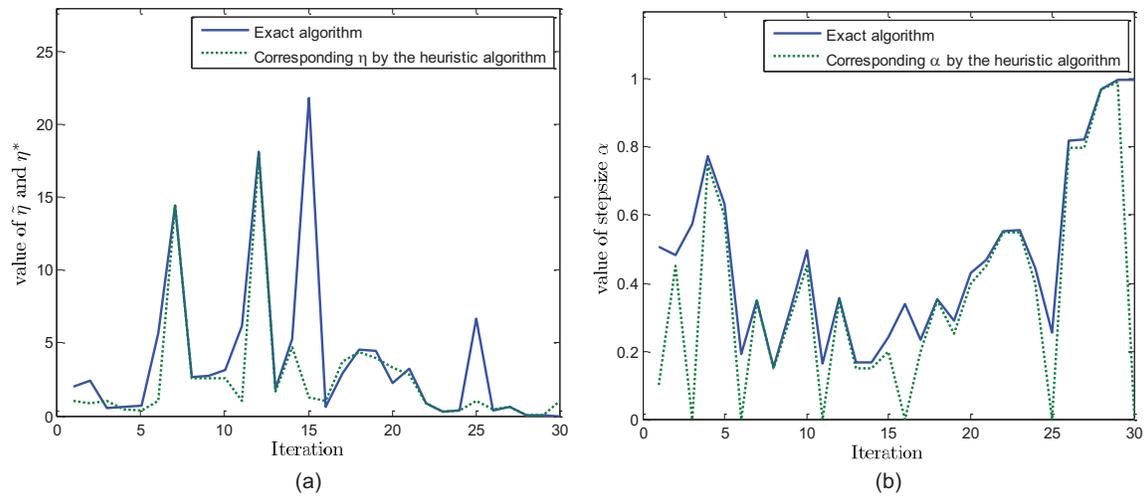}
\end{center}
\caption{\small Values of (a) $\eta$ (b) $\alpha$ for the exact algorithm, and the corresponding values calculated by the heuristic algorithm at each iteration of it, for problem \emph{ship08s}. \label{Fig-ship08s_m1}}
\end{figure}

\newpage

\appendix

%%%%%%%%%%%%%%%%%%%%5
\section{Connection with Kernel functions} \label{App1}
In this section, we introduce the Kernel function approach for interior-point methods \cite{kernel,lee2013kernel,cai-kernel} and discuss its connection with our approach. Let $\Psi(v) : \mathbb R^n_{++} \rightarrow \mathbb R$ be a strictly convex function such that $\Psi(v)$ is minimal at $v=e$ and $\Psi(e)=0$. In the Kernel function approach we replace the last equation of \eqref{KKT-2} with 
\begin{eqnarray} \label{Kernel-1}
Sd_x+Xd_s = -\sqrt{\mu} V \nabla \Psi \left (\frac{v}{\sqrt{\mu}} \right ),
\end{eqnarray}
where $v:= X^{1/2}S^{1/2}e$ \cite{kernel}. Note that by the definition of $\Psi$, $\nabla \Psi \left (\frac{v}{\sqrt{\mu}} \right ) = 0$ if and only if $(x,s)$ is on the central path. To simplify the matters, we assume that
\[
\Psi(v) = \sum_{j=1}^n \psi(v_j),
\]
where $\psi (t) : \mathbb R_{++} \rightarrow \mathbb R$ is an strictly convex function with unique minimizer at $t=1$ and $\psi(1)=0$. We call the univariate function $\psi(t)$ the Kernel function of $\Psi(v)$. It has been shown that the short update primal-dual path following algorithms using special Kernel functions obtain the current best iteration complexity bound \cite {kernel}. 

Comparing \eqref{KKT-3} and \eqref{Kernel-1}, we observe that both approaches are similar in the sense that the RHS of the last equation in \eqref{KKT-2} is replaced by a nonlinear function of $Xs$. The question here is whether there exists a  continuously differentiable strictly monotone function $f$ for each Kernel function $\psi$ or vice versa such that \eqref{KKT-2} and \eqref{Kernel-1} give the same search direction. In other words, can we solve 
\begin{eqnarray} \label{Kernel-2}
-\sqrt{\mu} t \psi'\left (\frac{t}{\sqrt{\mu}} \right ) = K \frac{ f(\mu) - f(t^2)}{f'(t^2)},
\end{eqnarray} 
for $f$ or $\psi$, for a constant scalar $K$? 
For $t=\sqrt{\mu}$, both sides of \eqref{Kernel-2} are equal to zero, so the equation is consistent in that sense. $\psi(t)$ is a strictly convex function with minimum at $t=1$, so $\psi'\left (\frac{t}{\sqrt{\mu}} \right ) < 0$ for $t < \sqrt{\mu}$ and $\psi'\left (\frac{t}{\sqrt{\mu}} \right ) > 0$ for $t > \sqrt{\mu}$. This makes both sides of \eqref{Kernel-2} consistent for a strictly monotone function $f$. Hence, \eqref{Kernel-2} may be solved for $f$ or $\psi$, however the result depends on $\mu$ in general. Table 2 shows five pairs of functions. Some of the Kernel functions in the table are from the functions studied in \cite{kernel}, and we solved \eqref{Kernel-2} for the corresponding $f(x)$. In the last two ones, we picked $f(x)=\ln(x)$ and $f(x)=\sqrt{x}$ and found the corresponding $\psi(t)$. 

\newpage 
\begin{center}
\text{Table 2: Some $\psi(t)$ and their corresponding $f(x)$ in view of \eqref{Kernel-2}.}\\
\begin{tabular} {| c  | c |} 
\hline
$\psi(t)$ & $f(x)$ \\
\hline
\hline
$\frac{t^2-1}{2}-\ln (t)$ & $x$ \\
\hline
$\frac 12 \left( t - \frac 1t \right)^2$  & $x^2$ \\
\hline
$\frac 12(t^2-1) + \frac{t^{-2q+2}-1}{-2q+2}, \ \ q>1$  & $x^q$ \\
\hline
$\frac 12 \left (t^2 + \frac{1}{t^2} \right) -1$  & $\frac 1x$ \\
\hline
$\frac{t^2-1}{2}+ \frac{t^{1-q}-1}{q-1}, \ \ q>1$ & $\ \ \ x^{\frac{-q+1}{2}} \ \ \ $ \\
\hline
$(t-1)^2$  & $\sqrt{x}$ \\
\hline
$t^2 \ln(t) - \frac 12 t^2 +\frac 12$ & $\ln(x)$ \\
\hline
\end{tabular}
\end{center}

As an example, we see the derivation of $f(x)$ for the third $\psi(t)$: we have $\psi'(t)=t-t^{-q}$, then
\begin{eqnarray*}
-\sqrt{\mu} t \psi'\left (\frac{t}{\sqrt{\mu}} \right ) &=& -\sqrt{\mu} t \left(\frac{t}{\sqrt{\mu}} -\frac{\mu^{-q/2}}{t^{-q}} \right) = -t^2+\frac{\mu^{\frac{-q+1}{2}}}{t^{-q-1}}  \\
&=& \frac{\mu^{\frac{-q+1}{2}}-t^{-q+1}}{t^{-q-1}} = 2 \frac{ f(\mu) - f(t^2)}{f'(t^2)}, \ \ f(x)=x^{\frac{-q+1}{2}}.
\end{eqnarray*}
For the forth pair, the function $\psi(t)=t^2 \ln(t) - \frac 12 t^2 + \frac 12$ obtains its minimum at $t=1$ with $\psi(1)=0$, and is decreasing before $t=1$ and increasing after that. The function is also convex around $t=1$, but it is not convex on the whole range of $t>0$.

As mentioned above, for each Kernel function $\psi(t)$, solving \eqref{Kernel-2} for $f(x)$ may result in a function depending on $\mu$. We can cover that by generalizing our method as follows. At each iteration, instead of applying $f(\cdot)$ to both sides of $Xs=\mu e$, we apply a function of $\mu$, i.e. $f(\mu;\cdot)$. The rationale behind it is that we expect different behaviours from the algorithm when $\mu >1$ and $\mu \ll 1$; e.g., we expect quadratic or at least super-linear convergence when $\mu \ll 1$. Hence, it is reasonable to apply a function $f(\cdot)$ that depends on $\mu$. We saw above that \eqref{Kernel-2} gives a non-convex function $\psi(t)$ for $f(x)=\ln(x)$ and the Kernel function approach does not cover our approach. However, our generalized method contains the Kernel function approach and is strictly more general in that sense. 

Consider \eqref{Kernel-2} for $K=\sqrt{\mu}/2$ and assume, without loss of generality, that $f(\sqrt{\mu})=0$. Then, from \eqref{Kernel-2},  for $t \neq\sqrt{\mu}$ we have:
\begin{eqnarray} \label{Kernel-3}
\frac{2tf'(t^2)}{ f(t^2)-f(\mu)} = \left (\psi'\left (\frac{t}{\sqrt{\mu}} \right ) \right)^{-1}  \ \ &\Rightarrow& \ \ \frac{d}{dt} [\ln(|f(t^2)-f(\mu)|)]=\left (\psi'\left (\frac{t}{\sqrt{\mu}} \right ) \right)^{-1}    \nonumber \\
  &\Rightarrow&  |f(t^2)| = \text{exp} \left [  \int \left (\psi'\left (\frac{t}{\sqrt{\mu}} \right ) \right)^{-1}  dt \right],
\end{eqnarray} 
where we have $f(t^2) < 0$ for $t < \sqrt{\mu}$ and $f(t^2) > 0$ for $t > \sqrt{\mu}$. As an example, consider the kernel function $\psi(t):=t-1+\frac{t^{1-q}-1}{q-1}$ (see \cite{kernel}) for the special case of $q=2$. Then we have $\psi'(t)=1-t^{-2}$. Substituting this in \eqref{Kernel-3}, we have:
\begin{eqnarray} \label{Kernel-4}
|f(t^2)| &=& \text{exp} \left [  \int \frac{t^2}{t^2 - \mu}  dt \right] =  \text{exp} \left [  \int 1+\frac{\mu}{t^2 - \mu}  dt \right]  \nonumber \\
                               &=&\text{exp} \left [  \int 1+\frac{\sqrt{\mu}}{2} \left ( \frac{1}{t-\sqrt{\mu}}- \frac {1}{t+\sqrt{\mu}} \right )  dt \right] \nonumber \\
                               &=&  e^t  \left (\frac{|t-\sqrt{\mu}|}{t+\sqrt{\mu}} \right)^{\frac{\sqrt{\mu}}{2}}.
\end{eqnarray} 
As can be seen, the concluded function $f(\cdot)$ is a function of $\mu$.

%%%%%%%%%%%%%%%%%%%%%%%%%%%%%%%%%%%%%%%5

\section{Homogeneous Self-Dual Embedding} \label{App2}
In this section, we introduce the homogeneous self-dual embedding \cite{YTM}.
We can construct a homogeneous and self-dual artificial LP problem
(HLP) related to (P) and (D) as follows: 
given any $x^{(0)}>0$, $s^{(0)}>0$, and $y^{(0)}$ free,
\small \begin{center}\begin{tabular}{rrrrrrlr}\\
\quad&min&\quad&\quad&\quad& $((x^{(0)})^{\top}s^{(0)}+1)\theta$&\quad\\
(1)&s.t.&\quad&$Ax$&$-bt$&$+\overline{b}\theta$& $=0$ &\quad\\
(2)&\quad&$-A^{\top}y$&\quad&$+ct$&$-\overline{c}\theta$&$\geq 0$ &\quad\\
(3)&\quad&$b^{\top}y$&$-c^{\top}x$&\quad&$+\overline{z}\theta $&$\geq 0$ &\quad\\
(4)&\quad&$-\overline{b}^{\top}y$&$+\overline{c}^{\top}x$&$-\overline{z}t
$&\quad &$=-((x^{(0)})^{\top}s^{(0)}+1)$ &\quad\\
(5)&\quad &y free,&$x \geq 0$,&$t \geq 0$,&$\theta$ free,&
\quad&\quad
\\
\end{tabular}
\end{center}
\normalsize where
$\overline{b}:=b-Ax^{(0)}$,\quad $\overline{c}:=c-A^{\top}y^{(0)}-s^{(0)}$, and
$\overline{z}:=c^{\top}x^{(0)}+1-b^{\top}y^{(0)}$.

The relationships (1)-(3), with $t=1$ and $\theta=0$, represent
primal and dual feasibility (with $x \geq 0$) and reversed weak
duality, so that all together they define the set of primal and dual optimal solutions. To
achieve feasibility for $x=x^{(0)}$ and $(y,s)=(y^{(0)},s^{(0)})$, the
artificial variable $\theta$ is added with appropriate coefficients
and constraint (4) is added to achieve self duality.
Denote by $s$ the slack vector for the inequality constraint (2) and
by $\kappa$ the slack scalar for the inequality constraint (3).
We can see that (HLP) is homogeneous and self-dual.

The following are the properties of the (HLP) model \cite{YTM}. 
\begin{itemize}
\item {The Dual of (HLP), denoted by(HLD),  has the same form as (HLP), i.e., (HLD) is simply
(HLP) with $(y,x,t,\theta)$ being replaced by $(y',x',t',\theta ')$.
Here $ y',x',t',\theta'$ make up the dual multiplier vector for
constraint (1), (2), (3), and (4), respectively.}
\item{ (HLP) has a strictly feasible point for every choice of $x^{(0)} >0, s^{(0)}>0$.}

\item { (HLP) has an optimal solution and its optimal solution set
is bounded.}

\item { The optimal value of (HLP) is zero, and for every feasible
point $(y,x,t,\theta,s,\kappa)$ we have:
\center {$((x^{(0)})^{\top}s^{(0)}+1)\theta=x^{\top}s+t\kappa$.}\\}
\item { There is an optimal solution
$(y^*,x^*,t^*,\theta^*=0,s^*,\kappa^*)$, such that:}

\begin{center}
 $\left(  \begin{array}{c}
   x^*+s^* \\
   t^*$+$\kappa^*
 \end{array} \right) >0$,

\end{center}
which we call a  \emph{strictly self-complementary solution}.
\end{itemize}
   If we choose $y^{(0)}:=0$, $x^{(0)}:=e$, and $s^{(0)}:=e$, then  (HLP)
becomes:
\begin{center}\begin{tabular}{rrrrrrrlr}
\quad&\quad&min&\quad&\quad&\quad& $(n+1)\theta$&\quad\\
(1)&\quad&s.t.&\quad&$Ax$&$-bt$&$+\overline{b}\theta$& $=0$ &\quad\\
(2)&\quad&\quad&$-A^{\top}y$&\quad&$+ct$&$-\overline{c}\theta$&$\geq 0$ &\quad\\
(3)&\quad&\quad&$b^{\top}y$&$-c^{\top}x$&\quad&$+\overline{z}\theta $&$\geq 0$ &\quad\\
(4)&\quad&\quad&$-\overline{b}^{\top}y$&$+\overline{c}^{\top}x$&$-\overline{z}t $&\quad &$=-(n+1)$ &\quad\\
(5)&\quad&\quad&$x \geq 0,$&$t \geq 0,$&\quad&\quad&\quad&\quad\\
\end{tabular}
\end{center}
where $\overline{b}:=b-Ae$,\quad $\overline{c}:=c-e$, and  $\overline{z}:=c^{\top}e+1$.

If we look at the solution of (HLP), we can solve the initial (LP) by using the theorem below. 
\begin{theorem}\textup{\cite{YTM}} Let
$(y^*,x^*,t^*,\theta^*=0,s^*,\kappa^*)$
be a strictly-self-complementary solution for \textup{(HLP)}.Then:
\begin{itemize}
\item{\textup{(P)} has an optimal  solution  if and only if
$t^*>0$. In this case, $(x^*/t^*)$ is an optimal
solution for \textup{(P)} and $(y^*/t^*,s^*/t^*)$ is
an optimal solution for \textup{(D)}};
\item{ if $t^*=0$, then $\kappa^*>0$, which
implies that $c^{\top}x^*-b^{\top}y^* < 0$, i.e., at least one of
$c^{\top}x^*$ and $-b^{\top}y^*$ is strictly less than 0. If
$c^{\top}x^* < 0$  then \textup{(D)} is infeasible; if $-b^{\top}y^* <
0$ then \textup{(P)} is infeasible; and if both $c^{\top}x^* < 0$ and
$-b^{\top}y^* < 0$ then both \textup{(P)} and \textup{(D)} are
infeasible.}
\end{itemize}
\end{theorem}
So, homogeneous and self-dual model can guarantee that we have a
strictly feasible solution to start most interior-point algorithms, and a strictly-self-complementary solution of the homogeneous self-dual 
embedding immediately solves both of the problems (P) and (D). In this context, ``solving an LP" means
determining exactly which of the three possibilities (given by the Fundamental Theorem of LP) holds and providing a 
succinct  certificate of the claim.

\section{Proofs of some theorems, lemmas, and propositions.} \label{proofs}

\noindent \textbf{Proof of Lemma \ref{lem:2.5}.}
\begin{proof}
Let $\beta \in [0,\frac{1}{4}]$, $(x, s)\in
{\mathcal N}_\infty(\beta)$, and $\xi_{ij}$ and $\zeta_{ij}$, $ij \in \{21,22\}$, as in the statement of the lemma. We define
\begin{eqnarray} \label{M11}
f_{ij}(u):=\Delta_{ij}-\xi_{ij} n\delta(u),
\,\,\,\,\,\,\,\,\,\, F_{ij}(u):=\zeta_{ij} n\delta(u)- \Delta_{ij}, \,\,\,\, \nonumber \\
\Omega:=\left \{u \in \mathbb{R}^n: e^{\top}u=n, \ (1-\beta)e \leq u \leq
(1+\beta) e \right\}.
\end{eqnarray}
Consider the following four optimization problems for $ij \in \{21,22\}$:
\begin{center}\begin{tabular}{cccccccc}
\quad&${\displaystyle
{\min_{u \in \Omega}F_{ij}(u)}}$&
$\quad$& \textup{and} &$\quad$&
${\displaystyle {\min_{u \in \Omega}f_{ij}(u)}}.$&\quad \\
\end{tabular}
\end{center}
To prove the lemma, it is sufficient to prove that the optimal objective values of these four problems are at least 0. We prove it for $\min_{u \in \Omega}f_{21}(u)$ and the proofs for the rest of them are similar. We have
\begin{eqnarray}  \label{eqA1}
\ \ \ \ \ \nabla f_{21}(u)=-\xi_{21} e+u +(2U-\xi_{21} I)\ln(u), \ \ \nabla^2 f_{21}(u)=3I-\xi_{21}  U^{-1}+2\textup{Diag}(\ln(u)). 
\end{eqnarray}
For $u >0$, $2\ln(u_j)+3-\frac{\xi_{21} }{u_j}$ is an increasing
function of $u_j$ for every $j \in \{1,\ldots,n\}$, and by our definition of $\xi_{21}$ we have
\[2\ln(1-\beta)+3-\frac{\xi_{21} }{1-\beta} = 0. \]
Hence, by \eqref{eqA1}, $\nabla^2 f_{21}(u)$ is
positive semidefinite over $\Omega$, which implies $f_{21}(u)$ is a convex function over $\Omega$. Let us write the Lagrangian function of the optimization problem $\min_{u \in \Omega}f_{21}(u)$:
\[
{\mathcal L}_{21}(u,\lambda_1,\lambda_2,\lambda_3)=f_{21}(u)- \lambda_1(n-e^{\top}u)
-\lambda_2^{\top}(u-(1-\beta)e)- \lambda_3^{\top}((1-\beta)e-u),
\]
where $\lambda_1 \in \mathbb{R}^n$, $\lambda_2 \in \mathbb{R}^n_+$,
$\lambda_3 \in \mathbb{R}^n_+$. Let us define
\[
u^*:=e, \ \ \lambda^*_1:=-1, \ \ \lambda^*_2:=\xi_{21}e, \ \ \lambda^*_3:=0.
\]
Then, we have $\nabla {\mathcal L}_{21} (u^*,\lambda^*_1,\lambda^*_2,\lambda^*_3)=0$ and $\nabla^2 {\mathcal L}_{21} (u^*,\lambda^*_1,\lambda^*_2,\lambda^*_3)=\nabla^2 f_{21}(u^*)$ is positive definite. Therefore, by second order sufficient conditions for optimality, $u^*=e$ is an optimal solution of $\min_{u \in \Omega}f_{21}(u)$ with optimal objective value of 0. 

\end{proof}

%%%%%%%%%%%%%%%%%%%%%%%%%%%%%%%%%%%%%%%%%%%%%%
\noindent \textbf{Proof of Lemma \ref{LL1}}
%%%%%%%%%%%%%%%%%%%%%%%%%%%%%%%%%%%%%%%%%%%%%555
\begin{proof}
For $(x,s)\in {\mathcal N}_2(\frac{1}{4})$,  the following condition guarantees that  $(x(\alpha),s(\alpha))\in
{\mathcal N}_2(\frac{1}{2})$.
 
\begin{eqnarray}  \label{eq111}
\sum_{j=1}^n\left [\frac{x_j(\alpha)
s_j(\alpha)}{(1-\alpha)\mu}-1\right ]^2 \leq \sum_{j=1}^n\left
(\frac{x_js_j}{\mu}-1\right )^2 +\frac{3}{16}
\end{eqnarray}
Solving \eqref{SD}, we have $(d_x)_j=\sqrt{x_j/s_j} (w_p)_j$ and $(d_s)_j=\sqrt{s_j/x_j}(w_q)_j$, for $j\in \{1,\ldots,n\}$. Using these, we have
\begin{eqnarray*}
x_j(\alpha) s_j(\alpha) &=& \left(x_j + \alpha \sqrt{\frac{x_j}{s_j}} (w_p)_j \right)\left(s_j + \alpha \sqrt{\frac{s_j}{x_j}} (w_q)_j \right) \\
&=& x_js_j + \alpha \sqrt{x_j s_j} (w) + \alpha^2 (w_p)_j (w_q)_j.
\end{eqnarray*}
Substituting this in the LHS of \eqref{eq111} and expanding it, we get
\begin{eqnarray*}
&&\sum_{j=1}^n \left (\frac{x_j(\alpha) s_j(\alpha)}{(1-\alpha)\mu}-1 \right )^2 
= \sum_{j=1}^n \left (u_j-1+\frac{\alpha
x_js_j}{(1-\alpha)\mu}(\delta-\ln (u_j)))+\frac{\alpha^2}{(1-\alpha)\mu}(w_p)_j(w_q)_j \right)^2\\
&& \ \ \ \ \ =\sum_{j=1}^n \left [(u_j-1)^2+\frac{\alpha^2(u_j)^2}{(1-\alpha)^2
}[\delta^2 +\ln^2 (u_j)-2 \delta
\ln (u_j)] \right .\\
&& \ \ \ \ \ +\frac{\alpha^4}{(1-\alpha)^2\mu^2}(w_p)_j^2(w_q)_j^2+2(u_j-1)\left (\frac{\alpha u_j}{(1-\alpha)}(\delta-\ln (u_j))+\frac{\alpha^2}{(1-\alpha)\mu}(w_p)_j(w_q)_j \right )\\
&& \ \ \ \ \ \left .+2\frac{\alpha^3x_js_j}{(1-\alpha)^2\mu^2}(\delta -\ln (u_j))(w_p)_j(w_q)_j \right],
\end{eqnarray*}
where we used $u_j = (x_js_j)/\mu$. By cancelling out $\sum_{j=1}^n (u_j-1)^2$ from both sides of \eqref{eq111} and multiplying both sides by $(1-\alpha)^2 \mu^2 $, we obtain the following equivalent inequality:  
\begin{eqnarray*}
&&{\alpha^2}(\delta^2\sum_{j=1}^n(x_js_j)^2+\Delta_{22} \mu^2-2\delta
\Delta_{21} \mu^2)+{\alpha^4}\sum_{j=1}^n(w_p)_j^2(w_q)_j^2 \\
&+& 2 \alpha (1-\alpha) \sum_{j=1}^n[x_j^2s_j^2\delta- \mu \delta
x_js_j-x_j^2s_j^2\ln(u_j)+\mu
x_js_j\ln (u_j)]\\
&+&2\alpha^2 (1-\alpha)
\sum_{j=1}^nx_js_j(w_p)_j(w_q)_j+2{\alpha^3}\sum_{j=1}^n\delta
x_js_j(w_p)_j(w_q)_j\\
&-&2{\alpha^3}\sum_{j=1}^nx_js_j(w_p)_j(w_q)_j
\ln(u_j) \leq \frac{3(1-\alpha)^2\mu^2}{16}.
\end{eqnarray*}
After expansion of the inequality above, we obtain the following inequality for the predictor  step of length $\alpha$, with the coefficients given in the statement.
\[ d_4\alpha^4+d_3\alpha^3+d_2\alpha^2+d_1\alpha+d_0 \leq 0, \]
\end{proof}

%%%%%%%%%%%%%%%%%%%%%%%%%%%%%%%%%%%%%%%%%%%%%%%5
\noindent \textbf{Proof of Lemma \ref{LL2}}
\begin{proof}
Let us define $\beta := 1/4$. 
For every $(x,s) \in {\mathcal N}_2 \left (\frac{1}{4} \right )$ we have
$\frac{3}{4} < u_j < \frac{5}{4}$ and 
$0 \leq \delta \leq 1$; more precisely, $\frac{1}{96n} \leq \frac{\beta^2}{6n} \leq \delta \leq \frac{\beta^2}{n} \leq  \frac{1}{16n}$ on the boundary of ${\mathcal N}_2 \left (\frac{1}{4} \right)$ (see Lemma \ref{lem:3.2}).
By Lemma 1 of \cite{Mizuno92}, we know that
$\|W_pw_q\|\leq \frac{\sqrt{2}}{4}\|r\|^2$, i.e., $\sqrt{\sum_{j=1}^n(w_p)_j^2(w_q)_j^2} \leq \frac{\sqrt{2}}{4}\sum_{j=1}^n r_j^2$,
where $r_j=-v_j+\delta v_j-v_j\ln \left (\frac{v_j^2}{\mu} \right)$, and 
$|r_j^2| \leq v_j^2 \left |-1+\delta-\ln \left (\frac{v_j^2}{\mu} \right) \right|^2 \leq
 \frac{25}{16}v_j^2$, where the second inequality is due to the facts that
$\left |\ln \left (\frac{v_j^2}{\mu} \right ) \right | \leq \frac{1}{4}$ and $0 \leq \delta \leq
\frac{1}{16n}$ within ${\mathcal N}_2(\frac{1}{4})$. 
We need to bound $d_4$, $d_3$, $d_2$, and $d_1$. By using Corollary \ref{coro-1} and the above results, we have
\begin{eqnarray*}
d_1 \leq  32 \delta \beta^2 n \mu^2  + 6 \mu^2   \leq 7 \mu^2,
\end{eqnarray*}
\beann d_4
=16\sum_{j=1}^n(w_p)_j^2(w_q)_j^2 \leq 2\left (\sum_{j=1}^n
r_j^2\right )^2 \leq \frac{625}{128} \left (\sum_{j=1}^n
x_js_j\right)^2 \leq \frac{625}{128}n^2\mu^2 \leq 5n^2\mu^2. \eeann
Within ${\mathcal N}_2(\frac{1}{4})$, we have $\ln(\frac{x_js_j}{\mu}) \leq
1$. Hence, by Cauchy-Schwarz inequality we have:
\begin{align*}
|B|&=\left |\sum_{j=1}^n x_js_j \ln \left(\frac{x_js_j}{\mu} \right)(w_p)_j(w_q)_j\right |\leq \sum_{j=1}^n x_js_j\left |\ln \left (\frac{x_js_j}{\mu} \right)(w_p)_j(w_q)_j\right | \\
&\leq \sum_{j=1}^n x_js_j|(w_p)_j(w_q)_j| \leq \sqrt{\sum_{j=1}^n(w_p)_j^2(w_q)_j^2}\sqrt{\sum_{j=1}^n x_j^2s_j^2}\\
& \leq \frac{3n\mu}{4}\sqrt{\frac{25}{16}n\mu^2}\leq
n^{\frac{3}{2}}\mu^2.\\
\end{align*}
\noindent Similarly, \beann |C| \leq \sum_{j=1}^n
x_js_j|(w_p)_j(w_q)_j| \leq n^{\frac{3}{2}}\mu^2. \eeann

\noindent Moreover, \beann C \leq \sum_{\{i: (w_p)_j(w_q)_j \geq
0\}} x_js_j(w_p)_j(w_q)_j \leq \sum_{\{i: (w_p)_j(w_q)_j \geq 0\}}
x_js_j\frac{x_js_j}{4} \leq \frac{25}{64} n\mu^2 \leq
\frac{n\mu^2}{2}. \eeann

\noindent Since $|\delta-1|<1$, we get:
\beann d_3=32(\delta-1)C- 32B<32|C|+32|B| \leq
64n^{\frac{3}{2}}\mu^2. \eeann

\noindent Using  Lemma \ref{lem:2.5}  and Corollary \ref{coro-1}, for every  $(x,s)
\in {\mathcal N}_2(\frac{1}{4})$,
we have $\Delta_{22} \leq \frac{5}{16}$ and $\Delta_{21} \leq \frac{9}{32}$. Thus,
\begin{eqnarray*}
d_2&=&16 \left (\delta^2\sum_{j=1}^n(x_js_j)^2+\Delta_{22} \mu^2-2\delta \Delta_{21} \mu^2 +2C -2\delta \sum_{j=1}^n x_j^2 s_j^2  +2\Delta_{21} \mu^2 \right)-9\mu^2\\
&\leq& 16 \left ( \frac{1}{16} \delta^2 n \mu^2 + \frac{5}{16} \mu^2 +0 + n \mu^2 +0 + \frac{9}{16} \mu^2 \right) -9 \mu^2 \\
&\leq & 20 n \mu^2. 
\end{eqnarray*}
\end{proof}

\newpage 
%%%%%%%%%%%%%%%%%%%%%%%%%%%%%%%%%%%
\noindent \textbf{Proof of Theorem \ref{thm-wide}.}
%%%%%%%%%%%%%%%%%%%%%%%%%%%%%%%%%%%%%%%%%%%%5
\begin{proof}
For each $j \in \{1,\ldots, n\}$, let $u_j:=\frac{x_js_j}{\mu}$.
Let us consider the search  direction $w $ in the wide neighbourhood
$\mathcal{N}_\infty^-(\frac{1}{2}) $. For the next iteration, to
stay in the same neighbourhood, the step length $\alpha$ should
satisfy the following condition for every $j \in \{1, 2, \dots, n \}$.

\begin{eqnarray}  \label{eq113}
u_j+\frac{\alpha}{1-\alpha}u_j(\delta\eta-\eta\ln(u_j))+\frac{\alpha^2}{1-\alpha}\frac{(w_p)_j(w_q)_j}{\mu}
\geq 1/2.
\end{eqnarray}
By Lemma 1 of \cite{MTY93} we have $|(w_p)_j(w_q)_j| \leq \| w(\eta)\|^2/4$. 
Since all of our discussion is within
$\mathcal{N}_\infty^-(\frac{1}{2}) $, we know that:
$\frac{x_js_j}{\mu}\geq \frac{1}{2}$, $j \in \{1, 2, \dots, n \}$.
We also deduce that $\frac{x_js_j}{\mu} \leq \frac{n+1}{2}$ from the fact $\sum_{j=1}^n x_js_j=n\mu$. So,
\[0 \leq \Delta_{12} \leq
\max_j \left \{ \ln^2 \left ( \frac{x_js_j}{\mu} \right ) \right \}\sum_{j=1}^n\frac{x_js_j}{\mu}=n\ln^2 \left (\frac{n+1}{2} \right), \ \ \ 
0\leq  \delta \leq \ln \left ( \frac{n+1}{2} \right ).\]

A sufficient condition for \eqref{eq113} to hold is the following inequality:
\beann
u_j+\frac{\alpha}{1-\alpha}u_j(\delta\eta-\eta\ln(u_j))-\frac{\alpha^2}{1-\alpha}\frac{n\mu+\eta^2\mu\Delta_{12}-n\mu\delta^2\eta^2}{4\mu}
\geq 1/2,\eeann 
i.e.,
\beann
4 \left (u_j-\frac{1}{2} \right )+4\alpha \left (u_j\delta\eta-u_j\eta\ln(u_j) -u_j+\frac{1}{2} \right )-\alpha^2(n+\eta^2\Delta_{12}-n\delta^2\eta^2)\geq  0.\eeann 
We will discuss three cases with respect to the magnitute of $u_j$.
If $u_j \geq 1$, a sufficient condition for \eqref{eq113} to hold is:
\beann
-\alpha^2 \left (n+\eta^2 n\ln^2 \left ( \frac{n+1}{2} \right )  \right )+1+4\alpha \left (u_j\delta\eta-\eta
u_j\ln(u_j)-u_j+\frac{1}{2} \right )+2 \left (u_j-\frac{1}{2} \right )
\geq 0,\eeann
i.e., \beann
-\alpha^2 \left (n+\eta^2 n\ln^2 \left ( \frac{n+1}{2} \right)  \right )+1+2 \left (u_j-\frac{1}{2} \right )\left[2\frac{u_j}{u_j-\frac{1}{2}}\alpha(\delta\eta-\eta
\ln(u_j)-1)+1\right ] + 2\alpha 
\geq 0.\eeann
\noindent  Let $\alpha=\frac{1}{4\eta\sqrt{n}\ln (n) }$, then the above inequality holds.

\noindent  If $1 > u_j \geq \frac{9}{16}$, a sufficient condition for \eqref{eq113} to hold is:
\beann -\alpha^2 \left (n+\eta^2 n\ln^2 \left ( \frac{n+1}{2} \right )  \right)-2\alpha +\frac{1}{4}\geq 0.\eeann\\
Let $\alpha=\frac{1}{16\eta\sqrt{n}\ln (n)}$, then the above inequality holds.

\noindent  If $\frac{1}{2} \leq  u_j \leq \frac{9}{16}$, a sufficient condition for \eqref{eq113} to hold is:
\beann -\alpha^2(n+\Delta_{12}\eta^2)+4\alpha u_j\delta\eta+4\alpha \left (\frac{1}{2}\eta \ln \left ( \frac{16}{9} \right ) -\frac{1}{16} \right)\geq 0,\eeann
for $\eta \geq 1$, the following inequality suffices:
\beann -\alpha^2(4n+\Delta_{12}\eta^2)+2\alpha\delta\eta+\frac{3}{4}\alpha\geq 0.  \eeann

Let $\alpha=\frac{1}{10\eta n\ln (n)}$ with $\eta \geq 1$, then the above inequality holds.

Hence, for the case $\eta=1$, we see that the constant step length of $\alpha=\frac{1}{10n \ln(n)}$
achieves the iteration complexity bound of $O(n\ln(n)\ln \left ( \frac{1}{\epsilon} \right ))$. Therefore, the same
iteration complexity bound holds for Algorithm \ref{wide-1}.
\end{proof}

%%%%%%%%%%%%%%%%%%%%%%%%%%%%%%%%%%55
\noindent \textbf{Proof of Theorem \ref{thm-wide-2}.}
\begin{proof}
As $u_j \geq \frac{1}{2}$ for all $j$, then for each $u_j \leq 1$ we have $0 \leq \delta - \ln (u_j)  \leq \delta +\ln (2)$. Let us define $J \subset \{1,\ldots,n\}$ as the set of indices for which $u_j \leq \frac{3}{4}$. By using this and \eqref{m1} we have
\begin{eqnarray} \label{m2}
\|w(\eta)\|^2
             =\sum_{j \notin J} x_js_j + \sum_{j \in J} x_js_j \left (-1 + \frac{\delta - \ln (u_j)}{\delta+\ln (2)} \right)^2 \leq \sum_{j=1}^n x_js_j = n \mu    
\end{eqnarray}
We must show that we still have enough reduction in the duality gap. First note that $\sum_{j=1}^n u_j =n$, and we have $\sum_{j \in J} u_j \leq \frac{3}{4}n$. This means $\sum_{j \in J} x_js_j \leq \frac{3}{4} \sum_{j=1}^n x_js_j$. Thus, we have:
\begin{eqnarray} \label{m3}
 x(\alpha)^{\top} s(\alpha) &=& (1-\alpha) x^{\top}s + \alpha  \sum_{j \in J } x_js_j \left (\frac{\delta-\ln (u_j)}{\delta +\ln (2)}  \right) \nonumber \\ &\leq& (1-\alpha) x^{\top}s + \frac{3}{4} \alpha x^{\top}s = \left [ 1- \frac 14 \alpha \right] x^{\top}s. 
\end{eqnarray}

We want $\alpha >0$ as large as possible such that $\frac{x_j(\alpha)s_j(\alpha)}{\mu(\alpha)} \geq \frac{1}{2}$ for all $j$. Using \eqref{m3} we have
\begin{eqnarray*} \label{m4}
\frac{x_j(\alpha)s_j(\alpha)}{\mu(\alpha)} \geq \frac{x_j(\alpha)s_j(\alpha)}{\left (1-\frac 14 \alpha \right )\mu}  \geq \frac{1}{2}.
\end{eqnarray*}
Using $|(w_p)_j(w_q)_j| \leq \| w(\eta)\|^2/4$ and \eqref{m2}, it is sufficient for $\alpha > 0$ to satisfy:
\begin{eqnarray} \label{m5}
&&-2 \left (1- \frac 14\alpha \right ) +4(1-\alpha)u_j - \alpha^2 n \geq 0, \ \ \ \ j \notin J,  \nonumber \\
&&-2 \left (1- \frac 14\alpha \right ) +4(1-\alpha)u_j+4\alpha u_j \eta (\delta-\ln (u_j)) - \alpha^2 n \geq 0, \ \ \ \ j \in J. 
\end{eqnarray}
For $j \notin J$, $u_j > \frac 34$ and \eqref{m5} is satisfied if 
\[
-2n\alpha^2 - 5 \alpha + 2 \geq 0.
\]
Clearly for $\alpha=\frac{1}{5n}$, this inequality is satisfied.

For $j \in J$, we have $u_j \in \left[\frac 12, \frac 34 \right ]$. We further split this case into two cases: $u_j \in \left[\frac 12, 0.55 \right ]$ and $u_j \in \left(0.55, \frac 34 \right ]$.
In both cases, we use the following inequality:
\begin{eqnarray} \label{m7}
\frac{\delta-\ln(u_j)}{\delta + \ln(2)} \geq \frac{-\ln(u_j)}{\ln(2)},  \ \ \ \forall \ u_j \geq \frac 12.
\end{eqnarray}
For $u_j \in \left(0.55, \frac 34 \right ]$, by using \eqref{m7}, \eqref{m5} is satisfied if
\[
-2n\alpha^2 - 3.19 \alpha + 0.4 \geq 0.
\]
For $\alpha=\frac{1}{10n}$, this inequality is satisfied. 

For $u_j \in \left[\frac 12, 0.55 \right ]$, by using \eqref{m7}, \eqref{m5} is satisfied if
\[
-2n\alpha^2 + 0.04 \alpha \geq 0.
\]
For $\alpha=\frac{1}{100n}$, this inequality is satisfied. 
We conclude that $\alpha=\frac{1}{100n}$ satisfies \eqref{m5} for all possible cases, and in a similar way to the proof of Theorem \ref{thm-wide}, we can conclude the desired iteration complexity bound.
\end{proof}
%%%%%%%%%%%%%%%%%%%%%%%%%%%%%%%%%%%%%555555
%%%%%%%%%%%%%%%%%%%%%%%%%%%%%%%%%%%%%%%%%%5
\end{document}